\documentclass[11pt]{amsart}
\usepackage{amsmath}
\usepackage{amssymb}
\usepackage{amsfonts}
\usepackage{amsthm}
\usepackage{mathabx}
\usepackage{mathrsfs}
\usepackage{dirtytalk}

\newcommand*\norm[1]{ \left\| #1 \right\| }

\newcommand*\Z{ \mathbb{Z} }
\newcommand*\Q{ \mathbb{Q} }
\newcommand*\R{ \mathbb{R} }
\newcommand*\C{ \mathbb{C} }
\newcommand*\N{ \mathbb{N} }

\newcommand*\T{\mathbb T }
\newcommand\set[1]{\left\{ #1 \right\}}

\theoremstyle{definition}

\newtheorem{Remark}{Remark}

\theoremstyle{theorem}

\newtheorem{mythm}{Theorem}[section]

\newtheorem{mylemma}{Lemma}[section]

\newtheorem{mycor}{Corollary}[section]

\title[Continuity of the Lyapunov exponent for analytic cocycles]{Continuity of the Lyapunov exponent for analytic multi-frequency quasiperiodic cocycles}
\author{Matthew Powell}
\address{Department of Mathematics, University of California, Irvine CA, 92717}
\date{\today}

\begin{document}
\maketitle

\begin{abstract}
It is known that the Lyapunov exponent of analytic 1-frequency quasiperiodic cocycles is continuous in cocycle and, when the frequency is irrational, jointly in cocycle and frequency. In this paper, we extend a result of Bourgain to show the same continuity result for {\it multifrequency} quasiperiodic $M(2,\C)$ cocycles. Our corollaries include applications to multifrequency Jacobi cocycles with periodic background potentials.
\end{abstract}

A cocycle is an example of a dynamical system which arises in ergodic theory. A cocycle is a dynamical system on a vector bundle which preserves the linear bundle structure and induces a measure-preserving dynamical system on the base. A particular class of examples are the quasiperiodic cocycles defined over $(\C^2, \T^d),$ with the underlying dynamical system given by a shift on the $d$-dimensional torus, $\T^d \simeq \R^d/\Z^d \simeq [0,1]^d.$  More concretely, let $M(2,\C)$ be the set of $2\times 2$ matrices with complex entries, let $\T^d$ denote the $d$-dimensional torus, and for $\omega \in \T^d,$ let $T_\omega: \T^d \to \T^d$ be given by $T_\omega x = x + \omega.$ We call $\omega$ the frequency of the shift. A $d$-dimensional quasiperiodic cocycle is a pair $(A,\omega)\in C(\T^d, M(2,\C)) \times \T^d$ when viewed as a linear skew product: $(A,\omega)$ acting on $\C^2 \times \T^d$ with
\begin{equation}\label{CocycleDef}
(A,\omega)(w,x) = (A(x)w, T_\omega x).
\end{equation} 
The cocycle iterates are given by
 \begin{equation}
 (A, \omega)^N = (A_N, N\omega),
 \end{equation}
 where
 \begin{equation}
 A_N(x,\omega) = \prod_{j = N-1}^0 A(x + j\omega).
 \end{equation}
 Throughout this paper, we use the term cocycle to describe $A_N(x,\omega).$
 
 The study of quasiperiodic cocycles has immediate applications to the study of one-dimensional (and quasi-one-dimensional) quasiperiodic Schr\"odinger operators:
 \begin{equation}(H_{x,\omega}\psi)(n) = \psi(n -1) + \psi(n + 1) + v(x + n\omega)\psi(n),\label{eq:SEqn}\end{equation}
 where $v: \T^d \to \R.$ It is well known that solutions to the eigenequation $H\psi = E\psi$ may be recovered using the transfer matrix:
 $$\prod_{j = N - 1}^0 \begin{pmatrix} E - v(x + j\omega) & -1\\ 1 & 0\end{pmatrix}.$$
 This describes a cocycle, with 
 $$A(x) = \begin{pmatrix} E - v(x) & -1 \\ 1 & 0\end{pmatrix}.$$
 Cocycles of the above form are typically called Schr\"odinger cocycles. 
 
 Schr\"odinger cocycles are always $SL(2,\R)$ cocycles. Besides generality, our motivation for considering $M(2,\C)$ cocycles are the quasiperiodic Jacobi cocycles (c.f. \cite{MarxJacobi, JitomirskayaMarxReview}). A quasiperiodic Jacobi operator on $L^2(\Z)$ is given by \eqref{eq:JacobiEqn}.
 We can define an analogous cocycle to study such operators:
\begin{equation}\label{eq:JacobiCocycle}A(x) = \begin{pmatrix}
(E - v(x))& -\overline{a(x - \omega)} \\
a(x) & 0
\end{pmatrix}
\end{equation}
The fundamental difference, though, is that, based on the choice of function $a,$ the cocycle need not be $SL(2,\R).$ In fact, it could have zero determinant somewhere. Another popular example of non-$SL(2,\R)$ cocycles are the transfer matrices associated with orthogonal polynomials on the unit circle with quasiperiodic Verblunsky coefficients (c.f. \cite{SimonOPUC}).


Of particular interest to us is the Lyapunov exponent. We begin by defining
\begin{equation}
L_N'(A,\omega,x) = \frac 1 N \ln\norm{A_N(x,\omega)}.
\end{equation}
Denote
\begin{equation}
L_N'(A,\omega) = \int_{\T^d} L_N'(A,\omega,x).
\end{equation}
It follows by subadditivity considerations that the limit
\begin{equation}
L'(A,\omega) = \lim_{N\to\infty} \int_{\T^d} L_N'(A,\omega,x) dx
\end{equation}
exists.
We call $L'(A,\omega)$ the (upper) Lyapunov exponent of the cocycle $(A,\omega).$ 

In this paper, we study continuity properties of $L'(A,\omega).$ This has been studied extensively for analytic Schr\"odinger cocycles both when $d = 1$ and $d >1,$ as well as non-identically singular $M(2,\C)$ cocycles with when $d = 1.$ Up until now, continuity for non-identically singular $M(2,\C)$ cocycles when $d > 1$ is known only when the frequency satisfies a Diophantine condition. In this paper, we improve this to include all frequencies by extending Bourgain's multifrequency $SL(2,\C)$ result to cover the general $M(2,\C)$ case.

Let us now proceed to provide the formal definitions necessary to state our main theorem. 

We are interested in {\it analytic quasiperiodic cocycles}. That is, quasiperiodic cocycles $(A, \omega),$ where $A$ is taken to be an analytic $M(2,\C)$-valued function on $\T^d$ with an analytic extension, continuous up to the boundary, to the complex strip $|\Im z_j| < \rho, \rho > 0,$ for all $1 \leq j \leq d.$ We denote the space of such $A$ by $C_\rho(\T^d, M(2,\C)).$ We put a natural metric on the space of these cocycles:
\begin{equation}
d((A,\omega), (B,\omega')) = \norm{A - B}_\rho + \norm{\omega - \omega'}_{\T^d},
\end{equation}
where
\begin{equation}
\norm{A - B}_\rho = \sup_{z: |\Im z_j| < \rho} |A(z) - B(z)|
\end{equation}
and $\norm{\omega - \omega'}_{\T^d}$ is the usual norm on $\R^d/\Z^d = \T^d.$ Note that any analytic function on $\T^d$ has an analytic extension to {\it some} complex strip $|\Im z_j| < \rho, \rho > 0.$ This observation may be used to define an inductive topology on the space of all analytic cocycles: $\bigcup_{\rho > 0} C_\rho.$
Moreover, we will assume $\det(A(x))$ is not identically zero. 

\begin{Remark}
Note that any $M(2,\C)$ cocycle $A_N(x)$ for which $\det(A(x))$ is not identically zero can be renormalized to form an $SL(2,\C)$ cocycle (see e.g.\cite{JitomirskayaKosloverSchulteis}), however the resulting cocycle may lose boundedness if $\det(A(x))$ has zeros. This is precisely the nature of difficulty when extending $SL(2,\C)$ results to the $M(2,\C)$ case.
\end{Remark}

The (upper) Lyapunov exponent is then defined as
\begin{equation}
L'(A,\omega) = \frac 1 N \int_{\T^d} \ln \norm{A_N(x,\omega)}dx.
\end{equation}
Note that, while $L'(A, \omega)$ need not be non-negative, the related object
\begin{equation}
L(A,\omega) = \lim_{N\to\infty} \int_{\T^d} L_N(\tilde A,\omega,x) dx,
\end{equation}
is, where $\tilde A \in SL(2,\C)$ is a renormalization of $A:$
\begin{equation}
\tilde A = \frac 1 {|\det A|^{1/2}} A.
\end{equation}
Moreover, $L_N$ and $L'_N$ are related by the following relation:
\begin{equation}
L_N(A,\omega) = L'_N(A,\omega) - \frac{1}{2} \int_{\T^d} \ln|\det(A(x))| dx.
\end{equation}
It follows that, when $\ln|\det(A(x))| \in L^1,$ both $L$ and $L'$ share the same regularity properties. In particular, if one is continuous, in some sense, then so is the other. 
Throughout, we will occasionally write $L_N(A, x)$ or $L_N(x)$ in place of $L_N(A,\omega, x),$ when there can be no ambiguity. Similarly, we will occasionally write $L_N(A)$ in place of $L_N(A,\omega)$ when $\omega$ is clear.

\begin{Remark}
We would like to make a note about a convention that we use. Throughout this paper, we use capital letters (e.g. $C, C',$ etc.) to denote constants which are sufficiently large, and lower-case letters to denote constants which are sufficiently small (e.g. $c, c',$ etc.). How large/small depends, unless otherwise specified, on the dimension, $d,$ and uniform measurements of the cocycle, $A.$ 
\end{Remark}

We prove the following:
\begin{mythm}\label{Thm:MyMainThm}
Suppose $\omega = (\omega_1,...,\omega_d) \in \T^d.$ Let $(A, \omega)$ be an analytic quasiperiodic $M(2,\C)$-cocycle.
Suppose, moreover, that $\det(A) \not\equiv 0.$ Then $L(A,\omega)$ satisfies the following.
\begin{enumerate}
	\item[(a)] \label{Thm:CocycleCont} $L(A,\omega)$ is continuous in $A$ for any $\omega \in \T^d.$ 
	\item[(b)] $L(A,\omega)$ is jointly continuous in $A$ and $\omega$ for $\omega$ such that $k \cdot \omega \not= 0$ for any $k \in \Z^d\backslash\set{0}.$ 
\end{enumerate}
\end{mythm}

\begin{Remark}
Analyticity is necessary for continuity of $L(A,\omega),$ in general. It is well-known that the Lyapunov exponent is discontinuous in $C^0$-topology at all non-uniformly hyperbolic cocycles. There are also examples in $C^r, 1 \leq r \leq \infty.$ Wang-You \cite{WangYou} constructed $SL(2,\C)$ cocycles which are $C^\infty$ with $L(A,\omega) > 0,$ yet may be approximated in $C^\infty$ topology by cocycles with zero Lyapunov exponent. Similarly, Jitomirskaya-Marx \cite{JitomirskayaMarxContinuity} constructed examples of $M(2,\C)$ cocycles which are discontinuous in $C^\infty$-topology.
\end{Remark}

\begin{Remark}
Part (b) of the above theorem is optimal, in the sense that there are examples of cocycles $A_0$ for which $L(A_0,\omega)$ is a discontinuous function of $\omega$ at frequencies such that $\norm{k\cdot\omega} = 0$ for some $0 \ne k \in \Z^d.$ Indeed, consider any $0 \ne k = (k_1,...,k_d) \in \Z^d.$ Let $\lambda(x) = e^{2\pi i k\cdot x}e^{-2\pi(k_1+\cdots+k_d)},$ and define
\begin{equation}
A_0(x) = \begin{pmatrix} e^{\lambda(x)} & 0 \\ 0 & e^{-\lambda(x)} \end{pmatrix}.
\end{equation}
This generates an analytic quasiperiodic cocycle, $(A_0,\omega).$ We can easily verify the following:
\begin{enumerate}
	\item If $\norm{k\cdot \omega} = 0,$ then $L(A_0, \omega) = \left(\frac 2 \pi\right) e^{-2\pi(k_1 + \cdots + k_d)};$
	\item If $\norm{k\cdot \omega} \ne 0,$ then $L(A_0, \omega) = 0.$
\end{enumerate}
Thus $L(A_0,\omega)$ is continuous at $(A_0, \omega)$ for any $\omega$ such that $\norm{k\cdot \omega} \ne 0$ and is discontinuous at $(A_0,\omega)$ for all $\omega$ such that $\norm{k\cdot \omega} = 0.$
\end{Remark}

The first result on continuity of $L$ for analytic cocycles is a theorem of Goldstein and Schlag \cite{GoldsteinSchlagAnnals}, who proved continuity in $E$ (in fact, they proved H\"older continuity) for Schr\"odinger cocycles with $d = 1,$ under the assumption that the frequency satisfies a strong Diophantine condition. The first result where (a) and (b) were established was \cite{BourgainJitomirskayaContinuity}, where it was done for 1-frequency $SL(2,\C)$ cocycles. See the next remark for a brief summary of the relevant historical developments of (a) and (b).

\begin{Remark}
\begin{enumerate}
\item When $d = 1,$ (a) and (b) were proved in \cite{BourgainJitomirskayaContinuity} for $SL(2,\C)$ cocycles, and later extended in \cite{JitomirskayaKosloverSchulteis} for non-identically singular $M(2,\C)$ cocycles under a Diophantine frequency assumption. 
\item When $d \geq 1,$ (a) and (b) were proved by Bourgain for Schr\"odinger cocycles (though the argument clearly applies to $SL(2,\C)$ cocycles) \cite{BourgainContinuity}.

\item When $d = 1,$ (a) was proven by Avila, Jitomirskaya, and Sadel for all analytic $M(n,\C),$ with any $n,$ cocycles \cite{AvilaJitomirskayaSadel} by a different method (see also \cite{JitomirskayaMarxContinuity}).

\item When $d \geq 1,$ (a) was proven by Duarte and Klein for all analytic $M(n,\C),$ with any $n,$ cocycles, assuming a fixed Diophantine frequency \cite{DuarteKleinBook, DuarteKleinSingular}. 
\item As far as we know, Theorem \ref{Thm:MyMainThm} is the first result establishing joint continuity for non-$SL(2,\C)$ cocycles, and continuity in the cocycle for {\it all} frequencies.
\end{enumerate}
\end{Remark}

We now describe an application to the special case of the Jacobi cocycle. The multifrequency analytic quasiperiodic Jacobi operator is defined as
\begin{equation}\label{eq:JacobiEqn}
(h_{x,\omega}\psi)(n) = \overline{a(x + (n - 1)\omega)}\psi(n -1) + a(x + n\omega)\psi(n + 1) + v(x + n\omega)\psi(n),
\end{equation}
where $v \in C_\rho^\omega(\T^d, \R)$ and $a \in C_\rho^\omega(\T^d,\C).$ Solutions to the eigenequation $H\psi = E\psi$ may be recovered using the transfer matrix:
 $$\prod_{j = N}^1 \begin{pmatrix} E - v(x + j\omega) & -\overline{a(x + (j - 1)\omega)}\\ a( x + j\omega) & 0\end{pmatrix}.$$
 The transfer matrix may be realized as a cocycle by setting 
 $$A(x) = \begin{pmatrix} E - v(x + \omega) & -\overline{a( x - \omega)} \\ a(x) & 0 \end{pmatrix}.$$
 These cocycles are singular precisely when $a(x)$ has zeros; we assume $a(x)$ does not vanish identically.
 The regularity of the Lyapunov exponent for such cocycles when $d = 1$ is already well-understood \cite{JitomirskayaKosloverSchulteis}. Consider, moreover, the quasiperiodic operator with a periodic background. That is, consider:
 \begin{equation}\label{eq:JacobiPeriodic}
(\tilde{h}_{x,\omega}\psi)(n) = (h_{x,\omega}\psi)(n) + v_{per}(n)\psi(n),
 \end{equation}
where $v_{per}$ is a $q$-periodic sequence of real numbers. Solutions to the eigenequation for this operator may be recovered from the new transfer matrix
$$\prod_{j = N }^1 \begin{pmatrix} E - v(x + j\omega) - v_{per}(j) & -\overline{a(x + (j - 1)\omega)}\\ a( x + j\omega) & 0\end{pmatrix}.$$
This transfer matrix may be realized as a quasiperiodic cocycle by ``regrouping along the period'' and setting
$$A(x) = \prod_{j = q }^1 \begin{pmatrix} E - v(x + j\omega) - v_{per}(j) & -\overline{a(x + (j - 1)\omega)}\\ a( x + j\omega) & 0\end{pmatrix}.$$
We can now define the Lyapunov exponent of this cocycle, $L(E, v, a, v_{per}, \omega),$ as usual, and it will, in fact, agree with the Lyapunov exponent associated with the transfer matrix. An immediate corollary of Theorem \ref{Thm:MyMainThm} is the following.

\begin{mycor}
Consider the multifrequency quasiperiodic Schr\"odinger operator with periodic background given by \eqref{eq:JacobiPeriodic}.
Suppose $v \in C_\rho^\omega(\T^d, \R)$ and $a \in C_\rho^\omega(\T^d,\C)$ do not vanish identically, and suppose $v_{per}$ is a $q$-periodic sequence of real numbers. Then we have the following.
 \begin{enumerate}
 \item $L(E, v, a, v_{per}, \omega)$ is continuous in $E, v, a$ and $v_{per}$ for any $\omega\in \T^d.$ 
 \item $L(E, v, a, v_{per}, \omega)$ is jointly continuous in $E, v, a, v_{per}$ and $\omega$ for $\omega$ such that $k \cdot \omega \not= 0$ for any $k \in \Z^d\backslash\set{0}.$ 
 \end{enumerate}
 \end{mycor}

Statements like Theorem \ref{Thm:MyMainThm}, for both $d = 1$ and $d > 1,$ have been studied extensively by many authors under suitable restrictions on the cocycle. Two particular methods have been used effectively in the past to establish continuity of $L(A,\omega)$ for quasiperiodic cocycles: one, used in \cite{AvilaJitomirskayaSadel} to prove continuity for arbitrary 1-frequency cocycles, is based on complexification of the cocycle and appealing to the notion of dominated splitting; recently, Duarte and Klein \cite{DuarteKleinObstructions} showed that there are large classes of multifrequency quasiperiodic cocycles which do not have dominated splitting, thus showing that this method cannot be used to address the multifrequency case; the second method (c.f. \cite{BourgainContinuity, BourgainJitomirskayaContinuity, DuarteKleinBook, JitomirskayaKosloverSchulteis}) is an induction scheme using the so-called Avalanche Principle and statistical properties of quasiperiodic cocycles. In this paper, we adapt the second method.


Bourgain and Jitomirskaya \cite{BourgainJitomirskayaContinuity} obtained joint continuity (as in (b)), a result that was essential for Avila's global theory, Ten Martini problem, and other important developments. It was observed by Jitomirskaya, Koslover, and Schulteis that this argument extends to the case of non-identically singular analytic cocycles which posses some analytic extension to a complex strip. The argument of these results relies on two ideas: first, a statistical property known as a large deviation theorem (LDT); and second, a general property of $SL(2,\C)$ matrices with large norm, known as the Avalanche Principle (AP). The argument of \cite{BourgainJitomirskayaContinuity} for $d = 1$ was based on the same basic ingredients as \cite{GoldsteinSchlagAnnals}: large deviation estimates LDT (i.e. statistical properties of the cocycle) and Avalanche Principle; however, \cite{BourgainJitomirskayaContinuity} constructed a special inductive scheme to deal with arbitrary frequencies and joint continuity.

A large deviation estimate is an estimate of the form:
$$\left|\set{x \in X: |f(x) - \int_X f(x)d\mu(x)| > \eta}\right| < \epsilon(\eta)$$
where, ideally, $\epsilon$ is exponentially small in $\eta.$
Such estimates were first used by Bourgain and Goldstein \cite{BourgainGoldstein} to establish Anderson localization for one-frequency quasiperiodic Schr\"odinger operators, and they have since been extended \cite{WencaiDisc} and play an important role in the study of various properties of quasiperiodic Schr\"odinger operators. For example, they have been used recently to obtain estimates on quantum dynamics (c.f. \cite{JitomirskayaPowell, LanaWencaiDynamics, ShamisSodin} etc.). 

The Avalanche principle was first introduced by Goldstein and Schlag \cite{GoldsteinSchlagAnnals} in their work on H\"older regularity of the integrated density of states, and variations of the original statement have been used in proofs of the continuity of $L(A,\omega)$ in various settings (c.f. \cite{AvilaJitomirskayaSadel, BourgainJitomirskayaContinuity, BourgainContinuity, JitomirskayaKosloverSchulteis}). 

The argument for $SL(2,\C)$ cocycles with $d = 1$ developed in \cite{BourgainJitomirskayaContinuity} roughly proceeds as follows. First, the frequency is assumed to be irrational, since rational frequencies are well-understood. Analyticity of the cocycle implies that $L_N(A,\omega,x)$ is subharmonic in $x$ with a {\it bounded} subharmonic extension to the strip $|\Im z| < \rho,$ for some $\rho > 0.$ An analysis of bounded subharmonic functions on the strip leads to estimates on the decay of the Fourier coefficients of $L_N,$ and this, in turn, leads to an LDT of the form
$$\left|\set{x \in \T: |L_N(A,\omega,x) - L_N(A,\omega)| > \kappa}\right| < e^{-c\kappa q},$$
where $\kappa$ and $q$ relate to properties of the frequency, $\omega.$ Combining this with the Avalanche Principle results in an estimate of the form
$$|L_{N_0}(A,\omega) - L_{N_1}(A,\omega)| < \kappa,$$ where $N_0$ is an initial scale which depends on measurements of the frequency, $\kappa$ is an error which depends on measurements of the frequency and $N_0,$ and $N_1$ is a multiple of $N_0$ which is not too large. This estimate is then used successively in an induction scheme to relate $L_{N_0}(A,\omega)$ to $L(A,\omega),$ and continuity of $L$ follows from continuity of $L_{N_0}.$

When $d > 1,$ serious technical issues arise which makes the arguments more complex. The only general result for arbitrary frequencies is Bourgain \cite{BourgainContinuity}, where an exact analogue of Theorem \ref{Thm:MyMainThm} was established for Schr\"odinger cocycles (though the argument extends without issue to $SL(2,\C)$ cocycles). One of the goals of this paper is to illustrate the power of the argument in \cite{BourgainContinuity} by extending it to a more difficult general $M(2,\C)$ case, while also providing additional details and clarifications to the original argument.

The rest of the paper is organized in the following way. In section \ref{Section:Outline} we briefly describe our argument. In Section \ref{Section:SHEstimates} we  recall the relevant facts about subharmonic and plurisubharmonic functions, and use them to prove two essential measure estimates, Lemma \ref{Lem:CDT} and Theorem \ref{Thm:LDT}. In Section \ref{Section:FiniteScaleCont}, we prove joint continuity of $L_N(A,\omega)$ for fixed $N$ and arbitrary $\omega.$ In Section \ref{Section:AP} we recall the Avalanche Principle and prove Theorem \ref{Thm:AvalanchePrincipleConsequence}, which we use throughout our induction scheme. In Section \ref{Section:Liouv}, we establish estimates between $L_N(A,\omega)$ at different scales when $\omega$ satisfies a Liouville-type condition. In Section \ref{Section:Mixed}, we establish estimates between $L_N(A,\omega)$ at different scales when $\omega$ satisfies a mixed Liouville-Diophantine condition. In Section \ref{Section:ContLemma} we use induction to extend the conclusions of Sections \ref{Section:Liouv} and \ref{Section:Mixed} to larger length scales. Finally, in Section \ref{Section:ContArg}, we use our induction result and finite-scale continuity to prove Theorem \ref{Thm:MyMainThm}. We also provide an appendix, where we provide proofs of the relevant plurisubharmonic function estimates from Section \ref{Section:SHEstimates}.

\section{A brief description of our argument}\label{Section:Outline}

In the $SL(2,\C)$ case, the major differences between $d = 1$ and $d > 1$ are largely a result of the interactions between the different components of the frequency. In particular, $L_N(A,\omega,x)$ is a {\it bounded} plurisubharmonic function (i.e. a multivariable function which is subharmonic in each variable) which does not behave as well as a subharmonic function. This makes the analysis necessary to obtain an LDT more technical. Moreover, there is no longer a dichotomy between rational and irrational frequencies, but rather a trichotomy between frequencies whose components are purely Diophantine, purely Liouville (or rationally dependent), and those with some components which are Diophantine and some which are Liouville (or rationally dependent). This complicates the argument in two ways. First, an LDT can only be obtained for purely Diophantine $\omega,$ so an additional argument is needed to obtain some (weaker) measure-theoretic estimate which is applicable when the frequency is not purely Diophantine. Second, the inductive procedure necessary to relate $L_{N_0}$ to $L$ is different depending on what kind of $\omega$ we have.

Let us now take some time to briefly describe our argument. We follow the same general structure as in \cite{BourgainContinuity}, and our argument can be viewed as an extension of Bourgain's. 
That is, the main scheme of our proof is adapted from \cite{BourgainContinuity}. However, while our result is significantly more general and more technically complex, our argument can be viewed as a clarification of Bourgain's main ideas, and, hopefully, improves the readability of the argument. In particular, we provide missing details in the arguments in Section \ref{Section:Liouv} and Section \ref{Section:ContLemma} and summarize the main ideas throughout.
The original argument, however, is not directly applicable in our general setting due to a few technical issues that arise while considering general cocycles. In particular, uniform (in $N$) pointwise boundedness and non-negativity of $L_N(x),$ as well as quantitative estimates on $|L_N(x) - L_N(x + \omega)|,$ are used extensively in Bourgain's work, while they no longer hold if $\det(A(x))$ is allowed to vanish, as, say, in the case of Jacobi cocycles; these need to be dealt with uniformly on all steps. 

Here, we give a brief description of Bourgain's scheme and the key difficulties in its adaptation.

The first step is to establish a large deviation estimate under suitable assumptions made on $\omega.$ As we noted above, this is typically arrived at by observing that $L_N(A,\omega,x)$ is plurisubharmonic with a {\it bounded} extension to a strip. Since we consider cocycles which may have singularities, $L_N(A,\omega,x)$ need not be bounded. Fortunately, following ideas introduced in \cite{DuarteKleinBook}, while we cannot say $L_N(A,\omega,x)$ is uniformly pointwise bounded, we can say it is uniformly $L^2$ bounded. It turns out that this is sufficient to perform the necessary analysis to obtain an LDT for $\omega$ which posses Diophantine-like properties up to suitably large scales (see Theorem \ref{Thm:LDT}). 

The uniform large deviation estimate we establish here (see Theorem \ref{Thm:LDT} is different from the uniform large deviation estimate established in \cite{DuarteKleinBook} (c.f. \cite{DuarteKleinBook} Theorem 6.6) in one crucial aspect: the result of Duarte and Klein requires an explicit Diophantine condition on the frequency, whereas our result requires a restricted Diophantine condition. In particular, every $\omega = (\omega_1,...,\omega_d) \in \T^d,$ where $\omega_1,...,\omega_d$ are irrational and rationally independent, satisfies a restricted Diophantine condition but need not satisfy a Diophantine condition. Due to the generality of the frequencies we consider, we lose any control over the modulus of continuity. It is this difference, however, which allows us to establish continuity which does not require a Diophantine assumption.

Our second step is to establish quantitative estimates on $|L_N(A, \omega, x) - L_N(A,\omega, x + a)|,$ which we use when our frequency is such that our LDT is not applicable. Our analysis of the plurisubharmonic function $L_N(A,\omega,x)$ allows us to say that, for any $a \in \T^d,$ $L_N(A,\omega,x)$ and $L_N(A,\omega, x + a)$ are close, away from a set of small measure (see Lemma \ref{Lem:CDT}). 

Next, we establish quantitative estimates on $|L_N(A, \omega, x) - L_N(A,\omega, x + \omega)|,$ which we use throughout to relate $L_{N_0}(A,\omega,x)$ and $L_{N_1}(A,\omega,x).$ In the Schr\"odinger case (and, in fact, in the case of nowhere singular cocycles) this is a simple consequence of the everywhere invertibility of $A;$ we consider cocycles which may be non-invertible somewhere. Once again, we are able to use our uniform $L^2$ boundedness, along with a uniform version of the Lojasiewicz inequality (see Lemma \ref{Lem:UnifLoj}), to prove that this difference is small away from an exponentially small set (see Lemma \ref{Lem:CDTVar}).

Next, we turn our attention to something which, at first glance, might seem trivial. In the $SL(2,\C)$ case, the entire argument relies on the well-understood fact that $L_N(A,\omega)$ is jointly continuous for any $(A,\omega)$ when $N$ is fixed. The typical argument for this relies on boundedness of $L_N(A,\omega).$ Since we no longer have boundedness, it is not immediately obvious why continuity should still hold. In the not-identically singular 1-frequency case \cite{JitomirskayaKosloverSchulteis}, joint continuity was proved using a cutoff argument and Diophantine considerations. In the multifrequency case with Diophantine frequency \cite{DuarteKleinBook}, continuity was proved using ergodicity considerations which required restrictive assumptions on the frequency. Neither method is wholly applicable in our setting, as we want a result for all frequencies. Using a uniform Lojasiewicz inequality, we are able to adapt the cutoff argument of Jitomirskaya, Koslover, and Schulteis and extend it to arbitrary frequencies.

Our next step is to establish our base estimate relating $L_{N_0}(A,\omega)$ to $L_{N_1}(A,\omega),$ for $N_1$ not too large, when $\omega$ is not Diophantine (see Theorem \ref{Thm:Liouv}. We prove this as a consequence of the Avalanche Principle and Lemma \ref{Lem:CDT}. This argument is of critical importance, as it provides the framework for estimates whenever the frequency is not purely Diophantine, and we appeal to it again when we prove Theorem \ref{Thm:Mixed}. 

We then turn our attention to relating $L_{N_0}(A,\omega)$ to $L_{N_1}(A,\omega),$ for $N_1$ not too large, when some components of $\omega$ are Diophantine and other components are not (see Theorem \ref{Thm:Mixed}). The case when the frequency is purely Diophantine is a special case of this. Our argument here relies on applying Lemma \ref{Lem:CDT} in the variables corresponding to the non-Diophantine components of $\omega$ and applying Theorem \ref{Thm:LDT} in those components which are Diophantine. This, eventually, leads us to a situation where the proof of Theorem \ref{Lem:Liouv} is applicable.

\begin{Remark}
Our estimates in Theorems \ref{Thm:Liouv} and \ref{Thm:Mixed} differ from the corresponding estimates in the $SL(2,\C)$ case (c.f. \cite{BourgainContinuity} Corollary 3.12 and Lemma 3.26) by a small power of $\kappa,$ which is a consequence of using a uniform $L^2$ estimate for $L_N(A,x),$ rather than a uniform pointwise bound. 
\end{Remark}

Our final step is an inductive argument allowing us to iterate our initial estimates to larger scales (see Theorem \ref{Thm:MainStep}). This relies on a delicate argument where we alternate between applying a toral automorphism (i.e. a change of variables which does not change the value of $L_N(A,\omega)$) and applying Theorem \ref{Thm:Mixed}.

The continuity of $L$ then follows from Theorem \ref{Thm:MainStep} and continuity of $L_N(A,\omega).$

\section{Plurisubharmonic functions and related estimates}\label{Section:SHEstimates}


In this section, we present the relevant facts and estimates related to plurisubharmonic functions defined on complex strips in $\C^d.$ The results here are based on results found in Chapter 6 of \cite{DuarteKleinBook} and we apply them to recover results from Section 1 of \cite{BourgainContinuity}. We present the statements of the main results here, mostly without proof. We provide proofs of Lemma \ref{Lem:LDTver1}, Theorem \ref{Thm:LDT}, and Lemma \ref{Lem:CDT}, as we will make repeated use of these results in later sections. A detailed discussion and proofs of the remaining results is provided in \ref{Appendix}.

One of the major obstacles to extending results about Lyapunov exponents for Schr\"odinger cocycles to general $M(2,\C)$ cocycles is the lack of uniform pointwise boundedness in the latter case. It turns out, however, that a uniform $L^p$ estimate is sufficient for our argument. The following lemma establishes such a uniform estimate.

\begin{mylemma}[\cite{DuarteKleinBook} Proposition 6.3] \label{Lem:UniformL2} Let $A \in C_\rho^\omega(\T^d, M(2,\C))$ with $\det(A)$ not identically 0. Then there are $\delta = \delta(A) > 0$ and $C = C(A) < \infty$ such that for any $B \in C_\rho^\omega(\T^d, M(2,\C)),$ with $\norm{B - A}_\rho < \delta,$ then 
\begin{equation}
\norm{L_n{(B)}}_{L^2(\T^d)} \leq C
\end{equation}
and
\begin{equation}
\norm{\ln|\det(B(x))|}_{L^2(\T^d)} \leq C.
\end{equation}
\end{mylemma}

The main application of this lemma will be in excising certain small sets of ``bad" points (where the pointwise bound is large) and showing that the integral of $L_N(x)$ over these bad sets is small by H\"older's inequality.

Another major obstacle is relating $L_N(x)$ to $L_N(x + \omega).$ In the $SL(2,\C)$ case, the well-known estimate
$$|L_N(x) - L_N(x + \omega)| < C \frac{1}{N}$$
holds. Such an estimate does not hold, in general, for non-invertible cocycles. However, it is possible to show that such an estimate holds for a large set of $x.$

\begin{mylemma}[\cite{DuarteKleinBook} Proposition 6.4] \label{Lem:CDTVar} Let $A \in C_\rho^\omega(\T^d, M(2,\C))$ with $\det(A)$ not identically 0. Then there are $\delta = \delta(A) > 0$ and $C = C(A) < \infty$ such that for any $0 < a < 1,$ if $B \in C_\rho^\omega(\T^d, M(2,\C))$ with $\norm{B - A}_\rho < \delta,$ then 
\begin{equation}
|L_N(B, x) - L_N(B, x + \omega)| \leq C N^{-a}
\end{equation}
holds for all $N \geq 1$ and for all $x \not\in F_N,$ where $|F_N| < e^{-N^{1 - a}}.$
\end{mylemma}

Both of these estimates relies on a uniform version of the Lojasiewicz inequality, which is of independent interest to us.

\begin{mylemma}[\cite{DuarteKleinBook} Lemma 6.1]\label{Lem:UnifLoj} Let $f(x) \in C_\rho^\omega(\T^d,\C)$ be such that $f(x)$ is not identically zero. Then ther are constants $\delta = \delta(f) > 0, S = S(f) < \infty,$ and $b = b(f) > 0$ such that if $g(x) \in C_\rho^\omega(\T^d,\C)$ with $\norm{g - f}_\rho < \delta,$ then 
\begin{equation}
\left|\set{x \in \T^d: |g(x)| < t}\right| < S t^b
\end{equation}
for all $t > 0.$
\end{mylemma}

With these lemmas in hand, we may proceed with our analysis of $L_N(A,x).$ Our next goal is to obtain finer control over $L_N(x),$ with the eventual hope of obtaining a large deviation estimate. Large deviation estimates for quasiperiodic cocycles typically arise from suitable decay of the associated Fourier coefficients, so we begin by controling the behavior of the Fourier coefficients.

\begin{Remark} The following three lemmas may be recovered via a synthesis of the statements and proofs in Chapter 6 from \cite{DuarteKleinBook}.
For convenience, we provide proofs of these three results in Appendix \ref{Appendix}.
\end{Remark}

\begin{mylemma}\label{Lem:FourierDecay}
Let $A \in C_\rho^\omega(\T^d, M(2,\C))$ with $\det(A)$ not identically 0. Then there are $\delta = \delta(A) > 0$ and $C = C(A, \rho) < \infty$ such that for any $B \in C_\rho^\omega(\T^d, M(2,\C))$ with $\norm{B - A}_\rho < \delta,$ 
\begin{equation}
\sum_{k\in \Z^d, |k| > K_0} |\hat{L}_n(B,k)|^2 \leq C \frac{1}{K_0}.
\end{equation}
\end{mylemma} 

Though we are interested in cocycles on $\T^d,$ with $d > 1,$ it is often possible to obtain results for $d > 1$ from the corresponding $d = 1$ result applied in each variable. We will occasionally use this technique when applying our Fourier coefficient estimate, so we also include the superior estimate we have when $d = 1.$

\begin{mylemma}
Let $A \in C_\rho^\omega(\T, M(2,\C))$ with $\det(A)$ not identically 0. Then there are $\delta = \delta(A) > 0$ and $C = C(A, \rho) < \infty$ such that for any $B \in C_\rho^\omega(\T, M(2,\C))$ with $\norm{B - A}_\rho < \delta,$ 
\begin{equation}
|\hat{L}_n(B,k)| \leq C \frac{1}{k}.
\end{equation}
\end{mylemma} 

We are now in a position to discuss a large deviation estimate. The general strategy is to combine the estimates from Lemma \ref{Lem:CDTVar} and Lemma \ref{Lem:FourierDecay} to obtain an $L^1$ estimate, which we then improve using the following fact about BMO (bounded mean oscillation) functions. We present this next estimate using $L_N(x),$ but it, in fact, holds for plurisubharmonic functions defined on strips in $\C^d$ which obey certain a priori estimates.

\begin{mylemma} \label{Lem:BMOIneq}
Let $A \in C_\rho^\omega(\T^d, M(2,\C))$ with $\det(A)$ not identically 0. Moreover, suppose 
\begin{equation}
\norm {L_N(A,x) - \int_{\T^d} L_N(A,x) dx}_{L^1} < \epsilon.
\end{equation}
Then there is $c = c(d)$ such that 
\begin{equation}
\left|\set{x \in \T^d: \left|L_N(A,x) - \int_{\T^d} L_N(A,x) dx\right| > \epsilon^c}\right| < e^{\epsilon^{-c}}.
\end{equation}
\end{mylemma}


These results allow us to obtain a uniform large deviation estimate for $L_N(A,x),$ which will be required in Section \ref{Section:Mixed}. We present the large deviation estimate in two steps, for clarity.


\begin{mylemma}\label{Lem:LDTver1}
Let $A \in C_\rho^\omega(\T^d, M(2,\C))$ with $\det(A)$ not identically 0.  Suppose $\omega \in \T^d$ is such that
$$\norm{k\cdot\omega} > \delta_0$$
for all $0 < |k| < K_0.$ Moreover, suppose 
$$R > \sqrt{K_0} \delta_0^{-1}.$$
Then there are $\delta = \delta(A) > 0$ and $C = C(A, \rho) < \infty$ such that for any $B \in C_\rho^\omega(\T^d, M(2,\C))$ with $\norm{B - A}_\rho < \delta,$ 
\begin{equation}
\left|\set{x \in \T^d: \left|\frac 1 R \sum_{j = 0}^{R - 1} L_N(B, x + j\omega) - \left\langle L_N(B) \right\rangle\right| > C_\rho K_0^{-c}}\right| < e^{-C_\rho K_0^c}.
\end{equation}
\end{mylemma}

\begin{proof}
Consider 
$$\left|\frac 1 R \sum_{j = 0}^{R-1} L_N(B, x + j\omega)  - \int_{\T^d}  L_n(B, x) dx \right|.$$ 
We have:
\begin{align}
\frac 1 R \sum_{j = 0}^{R-1} L_N(B, x + j\omega)  &= \frac 1 R \sum_{j = 0}^{R-1} \sum_{k \in \Z^d} \hat{L}_N(k)(B) e^{2\pi i k\cdot(x + j \omega)}\\
&= \frac 1 R \sum_{j = 0}^{R-1} \hat{L}_N(0)(B) + \frac 1 R \sum_{j = 0}^{R-1} \sum_{0 < |k| \leq K_0} \hat{L}_n(k)(B) e^{2\pi i k\cdot (x + j \omega)}\\
&\quad + \frac 1 R \sum_{j = 0}^{R-1} \sum_{|k| > K_0} \hat{L}_N(k)(B) e^{2\pi i k\cdot(x + j \omega)}\\
&= (I) + (II) + (III).
\end{align}
Observe that we have $$(I) = \int_{\T^d} L_N(B,x) dx.$$ Thus

$$(I) - \int_{\T^d} L_N(B,x) dx = 0,$$


Next, for $0 < |k| \leq K_0$ we may appeal to our condition on $\omega$ to conclude $$\left|\frac 1 R \sum_{j = 0}^{R-1} e^{2\pi i k \cdot j \omega}\right| \lesssim \frac{2}{R\norm{k\omega}} \leq 2 K_0^{-1/2}.$$ Thus
\begin{align}
\norm{(II)}_{L^2} &\leq  C K_0^{-1/2}.
\end{align} 
Here $C$ depends only on $A.$

Finally, for $|k| > K_0,$ we know $\sum_{|k|> K_0}|\hat{L}_N(k)(B)|^2 < C |K_0|^{-1},$ where $C$ depends only on $A.$ Moreover, $|e^{2\pi i j k\cdot \omega}| = 1,$ so
\begin{align}
\norm{ (III) }_2 &\leq \left(\sum_{|k| > K_0} |\hat{u}_n(k)|^2 \right)^{1/2} \\
&< C K_0^{-1/2}.
\end{align}

Hence $$\norm{\frac 1 R \sum_{j = 0}^{R-1} L_N(B, x + j\omega)  - \int_{\T^d} L_N(B, x) dx}_{L^1} < CK_0^{-1/2}.$$ 

Now we appeal to Lemma \ref{Lem:BMOIneq} applied to the function $\frac 1 R \sum_{j = 0}^{R - 1} L_N(B, x + j\omega) - \int_{\T^d} L_N(B, x),$ which completes our proof.
\end{proof}


\begin{mythm}\label{Thm:LDT}
Let $A \in C_\rho^\omega(\T^d, M(2,\C))$ with $\det(A)$ not identically 0.  Suppose $\omega \in \T^d$ is such that
$$\norm{k\cdot\omega} > \delta_0$$
for all $0 < |k| < K_0.$ Moreover, suppose 
$$N > K_0 \delta_0^{-1}.$$
Then there are $\delta = \delta(A) > 0$ and $C = C(A, \rho) < \infty$ such that for any $B \in C_\rho^\omega(\T^d, M(2,\C))$ with $\norm{B - A}_\rho < \delta,$ 
\begin{equation}
\left|\set{x \in \T^d: \left|L_N(B, x) - L_N(B) \right| > C_\rho K_0^{-c}}\right| < e^{-C_\rho K_0^c}.
\end{equation}
\end{mythm}

\begin{proof}

Apply Lemma \ref{Lem:LDTver1} with $R = \sqrt{N}.$ We obtain 
\begin{equation}
\left|\set{x \in \T^d: \left|\frac 1 {\sqrt{N}} \sum_{j = 0}^{\sqrt{N} - 1} L_N(B, x + j\omega) - \left\langle L_N(B) \right\rangle\right| > C_\rho K_0^{-c}}\right| < e^{-C_\rho K_0^c}.
\end{equation}

Moreover, recalling 
$$\left|\set{x: |L_N^j(B, x + j\omega) - L_N(B, x)| < C |j|N^{-1/2}}\right| < e^{-N^{1/2}},$$ 
where $C$ depends only on $A,$ away from a set of measure at most $R^2 e^{-N^{1/2}} < e^{-N^{1/3}},$ we have
\begin{align}
\frac 1 {\sqrt{N}} \sum_{j = 0}^{\sqrt{N}-1} L_N(B, x + j\omega) &= \frac 1 {\sqrt{N}} \sum_{j = 0}^{\sqrt{N}-1} \left( L_N(B, x) + O(|j|/N)\right)\\
&=  L_N(B, x)  + O(\sqrt{N}/N)\\
&\leq  L_N(B, x)  + CK_0^{-1/2}.
\end{align} 
Triangle inequality thus yields
$$\norm{L_N(B, x) - \int_{\T^d} L_N(B, x) }_{L^1} < CK_0^{-1/2}.$$
We now conclude in the same was as before.
\end{proof}

At this point, we have suitable estimates for frequencies which obey a diophantine estimate at certain length scales. We are interested, however, in general frequencies. The following estimate will be used in the absence of a diophantine frequency. This follows as a consequence of the Fourier coefficient decay estimate.
\begin{mylemma}\label{Lem:CDT}
Let $A \in C_\rho^\omega(\T^d, M(2,\C))$ with $\det(A)$ not identically 0. Then there are $\delta = \delta(A) > 0$ and $C = C(A, \rho) < \infty$ such that for any $B \in C_\rho^\omega(\T^d, M(2,\C))$ with $\norm{B - A}_\rho < \delta,$ and for $a\in \T$ small and $\kappa > 0,$ we have, uniformly in $N,$
\begin{equation}
\left|\set{x\in \T^d: |L_N(B,x) - L_N(B, x + a)| > \kappa}\right| < C\kappa^{-3}|a|
\end{equation}
\end{mylemma}

\begin{Remark} 
We will need $a$ to be small so that $C(A)^{-1}\kappa |a|^{-1} \geq 1.$ This is necessary for us to define $K_0$ appropriately in our proof (see \eqref{eq:K0DefCDT}).
\end{Remark}

\begin{proof}
We will prove this for $d = 1.$ The general case follows from the $d = 1$ case and Fubini's theorem.

We have
$$L_N(B,x) = \sum_{k \in \Z} \hat{L}_n(B,k) e^{2\pi k\cdot x},$$
so
\begin{align}
L_N(B,x) - L_N(B, x + a) &= \sum_{k \in \Z} \hat{L}_n(B,k) e^{2\pi k\cdot x} (1 - e^{2\pi k\cdot a})\\
&= \sum_{|k| < K_0} + \sum_{|k| \geq K_0}\\
&= (I) + (II).
\end{align}
Recall that, when $d = 1,$ we have
$$\left|\hat{L}_n(B,k)\right| \leq C(A) (1 + |k|)^{-1}.$$
It follows that
\begin{align}
|(I)| &\leq \sum_{|k| < K_0} \left|\hat{L}_n(B,k)\right| |(1 - e^{2\pi k\cdot a})|\\
&\leq \sum_{|k| < K_0} C(A) (1 + |k|)^{-1} |k||a|\\
&\leq C(A)K_0 |a|.
\end{align}

For $(II),$ we have
\begin{align}
\norm{(II)}_{L^2}^2 &\leq \sum_{|k| \geq K_0} 2 \left|\hat{L}_n(B,k)\right|^2\\
&\leq C(A) K_0^{-1}.
\end{align}

Taking 
\begin{equation}
K_0 \sim C(A)^{-1}\kappa |a|^{-1},\label{eq:K0DefCDT}
\end{equation}
we have
\begin{align}
|(I)| &\leq \kappa\\
\norm{(II)}_{L^2}^2 &\leq C(A) \kappa^{-1} |a|.
\end{align}
Applying Chebyschev's inequality,
\begin{equation}
\left|\set{x\in \T: |(II)| > \kappa}\right| < C(A)\kappa^{-3}|a|.
\end{equation}
It follows that
\begin{align}
\big|\big\{x: &|L_N(B,x) - L_N(B, x + a)| > \kappa\big\}\big| \\
&\leq \left|\set{x: |\sum_{|k| < K_0}| > \kappa}\right| + \left|\set{x: |\sum_{|k| \geq K_0}| > \kappa}\right|\\
&\leq C(A)\kappa^{-3}|a|.
\end{align}

\end{proof}

\section{Finite-scale continuity}\label{Section:FiniteScaleCont}


In the Schr\"odinger cocycle case (and the $SL(2\C)$ case more generally), one of the key observations is that, for fixed $N,$ $L_N(A,\omega)$ is jointly continuous in $A$ and $\omega$ for any $\omega.$ This is a simple consequence of the everywhere invertibility of the cocycle, $A.$ The analogous result for $M(2,\C)$ cocycles requires an argument.

In this section, we establish continuity of the finite-scale Lyapunov exponents, $L_N(A,\omega)$ jointly in $A$ and $\omega$ for any not identically singular $A$ and any frequency $\omega.$

\begin{mylemma}
Let $(A,\omega) \in C_\rho(\T^d, M(2,\C)) \times \T^d, A \not\equiv 0,$ be an analytic quasiperiodic cocycle. Then for any $\epsilon > 0,$ there exists constants $\delta = \delta(A,\epsilon), C = C(A,\epsilon),$ and $N_0 = N_0(A,\epsilon),$ such that for any $B \in C_\rho(\T^d,M(2,\C))$ with $\norm{A - B}_\rho < \delta,$ we have
$$|L_N(A,\omega) - L_N(B,\omega)| < C_A\epsilon$$
for all $N > N_0.$
\end{mylemma}

\begin{proof}

Let $\delta_0 > 0,$ fix $N > N_0,$ and set
\begin{align}
F_{A,\delta_0} &:= \set{x\in \T^d: \norm{A_N(x)} < e^{-N^{1 + \delta_0}}}\\
F_{B,\delta_0} &:= \set{x\in \T^d: \norm{B_N(x)} < e^{-N^{1 + \delta_0}}}\\
G &:= \set{x\in \T^d: \norm{A_N(x)} \leq \norm{B_N(x)}}.
\end{align}
Note that
\begin{align}
|L_N(A,\omega) - L_N(B,\omega)| &= \left|\int_{\T^d} \frac 1 N \ln \left(\frac{\norm{A_N(x)}}{\norm{B_N(x)}}\right)dx\right|\\
&= \left|\int_{F_A\cap F_B} + \int_{F_A^c \cap F_B} + \int_{F_A \cap F_B^c} + \int_{F_A^c \cap F_B^c}\right|.
\end{align}

Observe, for $x\in F_A^c \cap F_B^c,$ we have
\begin{equation}
\left|L_N(A,x) - L_N(B,x)\right| = \begin{cases} \frac 1 N \ln \left(\frac{\norm{A_N(x)}}{\norm{B_N(x)}}\right) & x\not\in G\\
\frac 1 N \ln \left(\frac{\norm{B_N(x)}}{\norm{A_N(x)}}\right) & x\in G
\end{cases}.
\end{equation}
Consider $x \in G.$ The case $x\not\in G$ will be the same. We have
\begin{align}
\frac 1 N \ln \left(\frac{\norm{B_N(x)}}{\norm{A_N(x)}}\right) &\leq \frac1 N \norm{A_N(x)}^{-1} C_A^N \norm{A - B}_\rho\\
&\leq e^{N^{1 + \delta_0}}C_A^N\norm{A - B}_\rho.
\end{align}
Taking $\norm{A - B}_\rho$ sufficiently small (dependent on $A, N,$ and $\delta_0)$ensures this is no more than $\epsilon/4.$

Next, consider $x \in F_A\cap F_B^c.$ The case $x \in F_A^c \cap F_B$ is the same. We have, necessarily, $x \in G,$ and thus
\begin{align}
|L_N(A,x) - L_N(B,x)| &= 1 N \ln \left(\frac{\norm{B_N(x)}}{\norm{A_N(x)}}\right)\\
&= \frac1N \left(\ln\norm{B_N(x)} - \ln\norm{A_N(x)}\right)\\
&\leq \frac1N\left(\ln\norm{B_N(x)} - \sum_{j = 1}^N \ln|\det A(x + j\omega)|\right)\\
&\leq \frac1N\left(NC_A - \sum_{j = 1}^N \ln|\det A(x + j\omega)|\right).
\end{align}
Moreover, $$F_A\cap F_B^c \subset F_A \subset \bigcup_{j = 1}^N \set{x: |\det(x + j\omega)| < e^{-N^{\delta_0}}} =: \bigcup_{j = 1}^N S_j.$$ Thus
\begin{align}
\left|\int_{F_A \cap F_B^c}\right| &\leq \sum_{j = 1}^N \frac1N \int_{S_j}\left(NC_A - \sum_{k = 1}^N \ln|\det A(x + k\omega)|\right)dx\\
&\leq C_A \sum_{j = 1}^N |S_j| + \frac1N \sum_{j = 1}^N\sum_{k = 1}^N\left|\int_{S_j}\ln|\det A(x + k\omega)|dx\right|\\
&\leq NC_Ae^{-\sigma N^{\delta_0}} + \frac1N \sum_{j = 1}^N\sum_{k = 1}^N\left|\int_{S_j}\ln|\det A(x + k\omega)|dx\right|.
\end{align}
Now, we may use Lemmas \ref{Lem:UniformL2} and \ref{Lem:UnifLoj} to bound
$$\frac1N \sum_{j = 1}^N\sum_{k = 1}^N\left|\int_{S_j}\ln|\det A(x + k\omega)|dx\right| \leq Ce^{-\sigma N^{\delta_0}}N^{1 + \delta_0}.$$
Putting all of this together and taking $N$ sufficiently large guarantees this is no larger than $\epsilon/4.$

The case $x \in F_A\cap F_B$ is similar.
\end{proof}


\section{Avalanche principle and immediate consequences}\label{Section:AP}


\begin{mythm}[Avalanche Principle] \label{Thm:AvalanchePrinciple}Suppose $A_1,...,A_n \in SL(2,\C)$ are such that
\begin{equation}
\min_{1 \leq j \leq n} \norm{A_j} \geq \mu > n
\end{equation}
and
\begin{equation}
\max_{1 \leq j < n}\left|\ln\norm{A)_j} + \norm{A_{j + 1}} - \ln\norm{A_{j + 1}A_j}\right| \leq \frac12 \ln \mu.
\end{equation}
Then 
\begin{equation}
\left|\ln \norm{A_n \cdot A_1} + \sum_{j = 2}^{n - 1} \ln\norm{A_j} - \sum_{j = 1}^{n - 1} \ln\norm{A_{j + 1}A_j}\right| < C \frac n \mu.
\end{equation}
\end{mythm}

The $C$ above is an absolute constant. We include this result for completeness, but we will actually use a slight variation of this result which is due to Bourgain \cite{BourgainContinuity}.

\begin{mythm}[Avalanche Principle Variation] \label{Thm:AvalanchePrincipleVar}
Suppose $A_1,...,A_n \in SL_2(\C)$ are such that
\begin{equation}
\mu < \norm{A_j} < \mu^C
\end{equation} 
for all $1 \leq j \leq n$ and some $\mu$ sufficiently large. Moreover, suppose
\begin{equation}
\max_{1 \leq j < n}\left|\ln\norm{A_j} + \norm{A_{j + 1}} - \ln\norm{A_{j + 1}A_j}\right| \leq \frac12 \ln \mu.
\end{equation}
Then
\begin{equation}
\left|\ln \norm{A_n \cdot A_1} + \sum_{j = 1}^{n} \ln\norm{A_j} - \sum_{j = 1}^{n - 1} \ln\norm{A_{j + 1}A_j}\right| < \frac{n}{\mu^{1/3}} + 4C \ln\mu.
\end{equation}
\end{mythm}

We refer readers to \cite{GoldsteinSchlagAnnals} for the proof of Theorem \ref{Thm:AvalanchePrinciple} and to \cite{BourgainContinuity} for the proof of Theorem \ref{Thm:AvalanchePrincipleVar}. Both of these references prove the Avalanche Principle for $SL(2,\R)$ matrices, but the arguments clearly apply to $SL(2,\C)$ matrices.

The Avalanche Principle allows us to relate Lyapunov exponents at some initial scale to Lyapunov exponents at a larger scale via the following.

\begin{mythm}\label{Thm:AvalanchePrincipleConsequence}
Fix an analytic cocycle $(A,\omega).$ Fix $x \in \T^d$ and let $\delta > 0$ be a fixed constant. Let $N_0 > 0$ be sufficiently large (depending only on $\delta$ and measurements of $A(x)$). Take $N_1 \in \Z, N_1 \geq N_0, N_0 | N_1.$ Assume, moreover, that
\begin{align}
L_{N_0}(x) &> \delta\label{eq:169}\\
|L_{N_0}(x) - L_{2N_0}(x)| &< \frac 1 {100} L_{N_0}(x)\label{eq:170}\\
|L_N(x) - L_N(x + jN_0\omega)| &< \frac{\delta}{100}\label{eq:171}
\end{align}
for $N = N_0, 2N_0,$ and $j \leq \frac{N_1}{N_0}.$ Then
\begin{align}\label{eq:FirstAPConcl}\begin{split}
\left|L_{N_1}(x) + \frac 1 n \sum_{j = 0}^{n-1} L_{N_0}(x + jN_0\omega) - \frac 2 n \sum_{j = 0}^{n - 2} L_{2N_0}(x + jN_0\omega)\right|\\
 < \exp\set{-\frac{N_0}{4}L_{N_0}(x)} + C|L_{N_0}(x)|\frac{N_0}{N_1},
\end{split}\end{align}
where $n = N_1/N_0.$ Moreover,
\begin{equation}\label{eq:SecondAPConcl}
\left|L_{N_1}(x) + L_{N_0}(x) - 2L_{2N_0}(x)\right| < \frac{\delta}{20} + C |L_{N_0}(x)|\frac{N_0}{N_1}.
\end{equation}
The $C$ here is an absolute constant.
\end{mythm}

We will provide a proof of this result for convenience, though a proof for Schr\"odinger cocycles may be found in \cite{BourgainContinuity}. A key difference between our presentation and the presentation in \cite{BourgainContinuity} is the inclusion of $L_{N_0}(x)$ in the right hand side of \eqref{eq:FirstAPConcl} and \eqref{eq:SecondAPConcl}. This is due to the absence of a uniform pointwise bound on $L_{N_0}(x)$ in the case of general cocycles. This will not pose a problem later, as we have Lemma \ref{Lem:UniformL2} to deal with integrals of the right-hand side.

\begin{proof}
Fix a cocycle $A.$ We will write $L_{N_0}(x)$ in place of $L_{N_0}(A,x).$ Define $M_j = A_{N_0}(x + j N_0\omega) \in SL(2,\C).$ Our goal is to apply Theorem \ref{Thm:AvalanchePrincipleVar} to $M_j.$ By \eqref{eq:171} and \eqref{eq:169}, we obtain
\begin{align}\begin{split}
\frac{99}{100} L_{N_0}(x) &\leq L_{N_0}(x) - \delta/100 \\
&\leq L_{N_0}(x + jN_0\omega) \\
&\leq L_{N_0}(x) + \delta/100 \\
&\leq \frac{101}{100} L_{N_0}(x).
\end{split}\end{align}
The definition of $L_{N_0}$ now yields, for $j \leq N_1/N_0,$ 
\begin{equation}
\frac{99}{100} N_0 L_{N_0}(x) < \ln\norm{M_j} < \frac{101}{100} N_0 L_{N_0}(x).
\end{equation}
Setting 
$$\mu = e^{\frac{99}{100}N_0 L_{N_0}(x)},$$
we have
\begin{equation}\label{eq:176} \mu < \norm{M_j} < \mu^{101/99} = \mu^C.\end{equation}
Moreover,
$$M_{j + 1} M_j = A_{2N_0}(x + jN_0\omega).$$
Thus, for $j < N_1/N_0,$
\begin{align}
\bigg| \ln&\norm{M_{j + 1}} + \ln\norm{M_j} - \ln\norm{M_{j + 1}M_j}\bigg| \\
&= N_0 \left|L_{N_0}(x + (j+1)N_0\omega) + L_{N_0}(x + jN_0\omega) - 2 L_{2N_0}(x + jN_0\omega)\right|.
\end{align}
Hence, by \eqref{eq:171} and triangle inequality,
\begin{align}
\bigg| \ln&\norm{M_{j + 1}} + \ln\norm{M_j} - \ln\norm{M_{j + 1}M_j}\bigg| \\
&\leq 2N_0 \left|L_{N_0}(x) + L_{2N_0}(x)\right| + \frac{N_0\delta}{25}\\
&\leq \frac{N_0\delta}{50} + \frac{N_0\delta}{25}\\
&= \frac{3N_0\delta}{50}\\
&< \frac{3N_0}{50} L_{N_0}(x)\\
&< \frac 1 {10} \ln\mu.
\end{align}
The last inequality follows from the definition of $\mu.$ This, along with \eqref{eq:176}, means Theorem \ref{Thm:AvalanchePrincipleVar} is applicable, and we obtain, with $N = \frac{N_1}{N_0} - 1,$
\begin{align}
\bigg| N_1 L_{N_1}(x) + &N_0\sum_{j = 0}^N L_{N_0}(x + j N_0\omega) - 2N_0 \sum_{j = 0}^{N - 1} L_{2N_0}(x + j N_0\omega)\bigg|\\
&< \frac{N_1}{N_0}\mu^{-1/3} + C \ln\mu\\
&= \frac{N_1}{N_0}\mu^{-1/3} + C N_0 L_{N_0}(x).
\end{align}
Dividing by $N_1,$ and using the definition of $N,$ we have
\begin{align}
\bigg| L_{N_1}(x) + &\frac 1 N (1 - N_0/N_1)\sum_{j = 0}^N L_{N_0}(x + j N_0\omega) \\
&- 2\frac 1 N (1 - N_0/N_1) \sum_{j = 0}^{N - 1} L_{2N_0}(x + j N_0\omega)\bigg|\\
&\quad = \frac{1}{N_0}\mu^{-1/3} + C \frac{N_0}{N_1} L_{N_0}(x).
\end{align}
Now we note that $$\frac 1 N (1 - N_0/N_1) = \frac{1}{N+1}$$ and $$N_0^{-1}\mu^{-1/3} < e^{-\frac 14 N_0 L_{N_0}(x)},$$ which together yield \eqref{eq:FirstAPConcl}. Combining this with \eqref{eq:171} and triangle inequality yields \eqref{eq:SecondAPConcl}.
\end{proof}





\section{Comparing Lyapunov exponents at different scales: Liouville frequencies} \label{Section:Liouv}


Throughout this section, we establish estimates of the form $|L_{N_0}(A) - L_{N_1}(A)| < C'\kappa.$ The constant $C'$ depends on measurements of $A$ and can thus be taken uniform for all $B$ with $\norm{A - B}$ sufficiently small. The uniformity of this constant will be essential to establishing continuity in Section \ref{Section:ContArg}. In what follows, the constant $C$ will be taken to be sufficiently large so that Theorem \ref{Thm:LDT} is applicable with $c = 1/C.$ 

Before considering the general case of arbitrary $\omega \in \T^d,$ we will consider the special case where $\omega$ satisfies a Liouville-type condition.

\begin{Remark}
We would like to note that we could just as easily have skipped the discussion in this section and simply proved Theorem \ref{Thm:Mixed}. We choose to present the following special case to illustrate the main ideas of the proof in a simplified setting.
\end{Remark}


\begin{mylemma}\label{Lem:BaseEstAP}
Fix $N_0 = a 2^bq_0,$ with $b\in \N$ large. Consider the set $F$ consisting of all $x\in \T^d$ such that $|L_N(x) - L_N(x + jq_0\omega)| < \kappa$ for all $N = 2^{-s}N_0,$ with $0 \leq s \leq -C_3\ln\kappa = s_0,$ and all $j \leq N_1/q_0$ such that $N_02^{-s}/q_0$ divides $j$ for some $0 \leq s \leq s_0.$ Here $C_3$ is a sufficiently large constant. Then 
$$\left|\int_F L_{N_0}(x) - L_{N_1}(x) dx\right| < C\kappa.$$
\end{mylemma}

\begin{proof}
Consider $x \in F$ such that $L_{N_0}(x) > 10^3\kappa.$ Define $N_{0,1} = N_0/2.$ Then $N_0 = 2N_{0,1}.$  Since $x\in F,$ we have
\begin{align}
L_{N_0}(x) &= \frac 1 {N_0} \ln\norm{A_{N_0}(x,\omega)}\\
&= \frac{1}{2N_{0,1}} \ln\norm{A_{2N_{0,1}}(x,\omega)}\\
&\leq \frac{1}{2N_{0,1}} \left(\ln\norm{A_{N_{0,1}}(x,\omega)} + \ln\norm{A_{N_{0,1}}(x + N_{0,1},\omega)}\right)\\
&= \frac12 L_{N_{0,1}}(x) + \frac 12 L_{N_{0,1}}(x + \frac{N_{0,1}}{q_0}q_0\omega)\\
&\leq L_{N_{0,1}}(x) + \kappa.
\end{align} 
The last line follows from the definition of $F.$ Thus 
$$L_{N_0} \leq L_{N_0/2}(x) + \kappa.$$

We can obtain a similar estimate using $N_{0,s} = N_0/2^s$ instead:
$$L_{N_02^{-s}}(x) \leq L_{N_02^{-s'}}(x) + \kappa$$
for any $s \leq s' \leq s_0.$

Since 
$L_{N_0}(x) > 10^3 \kappa,$ we have
\begin{equation}
999\kappa < L_{N_02^{-s}}(x).
\end{equation}

Thus, for $0 \leq s \leq s_0,$ the sequence $999\kappa < L_{2^{-s}N_0}(x)$ is increasing, up to modification by $O(\kappa).$

At this point, we further restrict our allowable $x$ to a suitable set $$Z = \set{x\in \T^d: L_{2^{-s_0}N_0}(x) < \kappa^{-1}}.$$ By Chebyschev's inequality, $|\T^d \backslash Z| < C(A) \kappa,$ so 
$$\left|\int_{\T^d \backslash Z} L_{N_1}(x) - L_{N_0}(x) dx\right| < C \kappa.$$
For $x\in F \cap Z,$ we define $N_{00} = N_{00}(x) = \frac{N_0}{2^{s(x)}},$ where $0 \leq s(x) \leq s_0$ is chosen so that $N_{00} < \kappa N_0,$ (which is possible because $s_0 = -C_3\ln\kappa > -\ln\kappa),$ and
$$\frac{99}{100}L_{N_{00}}(x) < L_{2N_{00}}(x) < L_{N_{00}}(x) + 10\kappa.$$
The right inequality is true because $L_{N_02^{-s}}(x)$ is increasing up to modification by $O(\kappa),$ and the left inequality is true by taking $C_3$ large. Indeed, if the left inequality fails for all choices of $N_{00}(x),$ then we have, for $x\in F\cap Z,$
$$\left(\frac{99}{100}\right)^{-C_3\ln\kappa}\kappa^{-1} > \left(\frac{99}{100}\right)^{-C_3\ln\kappa} L_{2^{-s_0}N_0}(x) > L_0(x) > 10^3\kappa.$$
Taking a logarithm of the leftmost and rightmost terms, we quickly see that this is impossible for large $C_3.$ (Say $C_3 > 10$).

Hence
\begin{equation}
|L_{N_{00}}(x) - L_{2N_{00}}(x)| < \frac{1}{100} L_{N_{00}}(x).
\end{equation}
Now note that $x\in F$ and $N_{00}, 2N_{00}$ are of the form of the length scales included in the definition of $F,$ so for $N = N_{00}$ and $2N_{00},$ and for all $j = n \frac{N_0}{2^sq_0}, j \leq N_1/q_0, 0 \leq s \leq s_0,$ we have
\begin{equation}
\left|L_N(x + jq_0\omega) - L_N(x)\right| < \kappa.
\end{equation}
In particular, if we take $j = n = N_{00}q_0^{-1}, n \leq N_1/N_{00},$ we have
\begin{equation}
\left|L_N(x + nN_{00}\omega) - L_N(x)\right| < \kappa.
\end{equation}
Thus Theorem \ref{Thm:AvalanchePrincipleConsequence} is applicable, and we obtain
\begin{align}
\left|L_{N_1}(x) + L_{N_{00}}(x) - 2L_{2N_{00}}(x)\right| &< O(\kappa) + C L_{N_{00}}(x) \frac{N_{00}}{N_1}\label{eq:APApplication}\\
\left|L_{N_0}(x) + L_{N_{00}}(x) - 2L_{2N_{00}}(x)\right| &< O(\kappa) + C L_{N_{00}}(x) \frac{N_{00}}{N_0}.
\end{align}
Since $x \in F,$ we have $L_{N_{00}}(x) \leq L_{2^{-s_0}N_0}(x) + \kappa.$ Moreover, by construction, $\frac{N_{00}}{N_1}, \frac{N_{00}}{N_0} < \kappa.$ Thus, for $x \in F$ such that $L_{N_0}(x) > 10^3\kappa,$
\begin{equation}
\left|L_{N_0}(x) - L_{N_1}(x)\right| < (C + L_{2^{-s_0}N_{0}}(x)) \kappa \label{eq:96}.
\end{equation}
Moreover, we know $\int_{\T^d} |L_{2^{-s_0}N_0}(x)| dx < C(A),$ so the integral of the above is bounded by $C(A)\kappa.$ 
Finally, for $x \in F$ such that $L_{N_0}(x) < 10^3 \kappa,$ we have
$$L_{N_1}(x) \leq L_{N_0}(x) < 10^3 \kappa.$$
Thus
\begin{equation}
\left|\int_F L_{N_1}(x) - L_{N_0}(x) dx\right| < C(A) \kappa.
\end{equation}

\end{proof}


\begin{mylemma}[Liouville Frequencies]\label{Lem:Liouv}



Let $\kappa > 0$ be small, and let $N_0, q_0 \in \N$ be such that 
\begin{equation}
\norm{q_0\omega} = \sum\norm{q_0\omega_j} < \kappa^C\rho^4\frac{q_0}{N_0} \label{eq:LiouvCondition}
\end{equation}
and
\begin{equation}
N_0\kappa^C > q_0.\label{eq:N0LargeCondition}
\end{equation}
Then
\begin{equation}
\left|L_{N_1} - L_{N_0}\right| < C' \kappa^{1/2}
\end{equation}
for all $N_1$ such that $N_0 | N_1$ and 
\begin{equation}
N_1 < \kappa^{C/2}\rho^2 \sqrt{\frac{N_0q_0}{\norm{q_0\omega}}}.\label{eq:N1SmallCondition}
\end{equation}
Here $C = C(v,E)$ and can be taken uniform in $E$ whenever $E$ is restricted to a bounded set.
\end{mylemma}

\begin{Remark}
The following proof is essential to our overall argument, and elements of it will be used again to prove Theorem \ref{Thm:Mixed}. In particular, the introduction of and restriction to the set $F$ is essential. Once we restrict to $F,$ the following argument goes through without issues. The key difficulty is showing that $\T^d\backslash F$ has small measure. This is accomplished here using Lemma \ref{Lem:CDT}. Later it will be accomplished using more involved estimates.
\end{Remark}

\begin{proof}
We will first prove this result for $N_0$ such that 
\begin{equation}
N_0 = a 2^b q_0,\quad a,b\in \N, b > -\frac C2\ln\kappa,
\end{equation}
for some $C$ sufficiently large (this is the same $C$ which appears in the statement of the lemma). Note that this assumption on $b$ allows for $N_0$ slightly smaller than those imposed by the condition \eqref{eq:N0LargeCondition}. Then we will derive the general result for general $N_0$ satisfying \eqref{eq:N0LargeCondition}.

Fix $\kappa$ small. 
Consider the set of $x\in\T^d$ such that 
$$|L_N(x + jq_0\omega) - L_N(x)| < \kappa$$
for all $2^{-s_0}N_0 \leq N \leq N_0,$ where $s_0 = - C_3\ln\kappa,$ with
$C_3 > 4$ 
and set $C >0$ such that $C - C_3 > 4,$ $N$ is of the form $2^{-s}N_0,$ and $j \leq N_1/q_0$ is such that $N_0/2^sq_0$ divides $j$ for some $s \leq s_0.$ In other words, we want to consider the set $F$ defined as follows. Define
\begin{equation}
F_j^N = \set{x \in \T^d: |L_N(x + jq_0 \omega) - L_N(x)| < \kappa}
\end{equation}
and set 
\begin{equation}
F = \bigcap_{n = 0}^{s_0} \bigcap_{s = 0}^{s_0} \bigcap_{m = 1}^{N_12^s/N_0} F_{mN_0/2^sq_0}^{2^{-n}N_0}
\end{equation}
We have
$$L_{N_0} - L_{N_1} = \int_F (L_{N_0}(x) - L_{N_1}(x)) dx + \int_{\T^d\backslash F} (L_{N_0}(x) - L_{N_1}(x)) dx.$$
Observe that Lemma \ref{Lem:CDT} implies
$$\left|\T^d \backslash F_j^N\right| \leq C \kappa^{-3}\rho^{-3} |j|\norm{q_0\omega}.$$
Thus, setting $\alpha(s) = N_1 2^s/N_0,$ we have
\begin{align}
|\T^d \backslash F| &\leq \sum_{n = 0}^{s_0}\sum_{s = 0}^{s_0}\sum_{m = 1}^{\alpha(s)} C \kappa^{-3}\rho^{-3}\norm{q_0\omega} mN_0/2^sq_0\\
&\sim \sum_{n = 0}^{s_0}\sum_{s = 0}^{s_0}\kappa^{-3}\rho^{-3}\norm{q_0\omega}\frac{N_1^2}{N_0q_0} 2^{s}\\
&\sim \sum_{n = 0}^{s_0} \kappa^{-3}\rho^{-3} \norm{q_0\omega} \frac{N_1^2}{N_0q_0} 2^{s_0}\\
&= \sum_{n = 0}^{s_0} \kappa^{-3}\rho^{-3} \norm{q_0\omega} \frac{N_1^2}{N_0q_0} \kappa^{-C_3}\\
&\sim s_0 \kappa^{-3 - C_3}\rho^{-3} \norm{q_0\omega} \frac{N_1^2}{N_0q_0}\\
&\leq \kappa^{-3-C_3}\rho^{-3} \norm{q_0\omega} \frac{N_1^2}{N_0q_0}.
\end{align}
Thus, the set of $x$ excluded from $F$ has measure at most
$$\kappa^{-3-C_3}\rho^{-3} \norm{q_0\omega} \frac{N_1^2}{N_0q_0} < \kappa,$$
and we have
$$\left|\int_{\T^d\backslash F} (L_{N_0}(x) - L_{N_1}(x)) dx\right| \leq C\kappa^{1/2}$$ by Lemma \ref{Lem:UniformL2}.

It thus suffices to understand $\int_F (L_{N_0}(x) - L_{N_1}(x)) dx.$ 

Since the set $F$ we consider here satisfies the conditions of the set from Lemma \ref{Lem:BaseEstAP}, we conclude that 
$$\left|\int_F (L_{N_0}(x) - L_{N_1}(x)) dx\right| < C\kappa.$$

Our conclusion now follows, assuming 
\begin{equation}
N_0 = a 2^b q_0,\quad a,b\in \N, b > -C\ln\kappa.
\end{equation}


Now consider arbitrary $N_0$ such that \eqref{eq:LiouvCondition} and \eqref{eq:N0LargeCondition} hold. Let $N_0' = 2^b q_0$ be of the form considered above such that $N_0' \leq \sqrt{N_0}.$ 
Moreover, let $\alpha \in \N$ be such that $\alpha N_0' \leq N_0 \leq (\alpha + 1)N_0'.$ In particular, $N_0 = \alpha N_0' + O(N_0').$ We may now run the above argument using $\alpha N_0'$ as our initial scale. Let $F$ be the same set considered before, with respect to the initial scale $\alpha N_0'.$ Then $|\T^d\backslash F| < \kappa,$ and thus 
$$\left|\int_{\T^d\backslash F} L_{N_0}(x) - L_{N_1}(x) dx \right| < C\kappa$$ 
for all suitable $N_0 | N_1.$ Furthermore, on $F$ such that $L_{\alpha N_0'}(x) < 10^3 \kappa$ and 
$$L_{O(N_0')}(A, x + \alpha N_0'\omega) < N_0 \kappa,$$ 
we also have $L_{N_0}(x) < C\kappa$ and $L_{N_1}(x) < C\kappa.$ Since 
$$\left|\set{x:  L_{O(N_0')}(A, x + \alpha N_0'\omega) < \frac{N_0}{O(N_0')} \kappa}\right| < C\kappa,$$
by Chebyschev's inequality and our choice of $N_0',$ excluding this set only changes our final integral by a term of order $\kappa.$

On $F$ such that $L_{N_0'}(x) > 10^3 \kappa,$ we may define $N_{00}(x) < \kappa \alpha N_0' \leq \kappa N_0$ as before and note that there is $a \in \N$ such that $1/a < \kappa$ and
\begin{equation}
aN_{00} < N_0 < (a + 1) N_{00}.
\end{equation} 
Moreover, we have, by \eqref{eq:APApplication},
\begin{equation}
|L_{\alpha N_{0}'}(x) - L_{(\alpha+1)N_{0}'}(x)| < C \kappa^{1/2}.
\end{equation}

Next, we consider the set 
$$G = \set{x\in \T^d: \norm{\tilde A_{O(N_{0}')}(x + \alpha N_{0}'  \omega)} < e^{N_0\kappa}}.$$ 
Clearly, since $N_{0}' < \sqrt{N_0},$ and $N_0 > \kappa^{-C},$ with $C > 4,$ we may apply Chebyschev's inequality to obtain
$$|\T^d\backslash G| \leq C(A) \frac{N_{0}'}{N_0\kappa} < C\kappa.$$
It follows that $$\left|\int_{\T^d\backslash G} L_{N_0}(x) - L_{N_0'}(x) dx \right| < C\kappa.$$ 
On $G,$  $L_{O(N_{0}')}(A, x + \alpha N_{0}'\omega) < \frac{N_0}{O(N_0')}\kappa,$ so  
\begin{align}
L_{N_0}(x) &\leq \frac{\alpha N_0'}{N_0} L_{\alpha N_0'}(x) + \frac{O(N_0')}{N_0}L_{O(N_0')}(x + \alpha N_0' \omega)\\
&\leq L_{\alpha N_0'}(x) + \kappa.
\end{align}

We may similarly obtain
\begin{equation}
L_{(\alpha + 1)N_{0}'}(x) \leq L_{N_{0}}(x) + \kappa
\end{equation}
by considering a slightly different set $G'$ with $|\T^d\backslash G'| < \kappa$ still.
Thus
\begin{equation}
|L_{N_0}(x) - L_{\alpha N_{0}'}(x)| < C\kappa^{1/2}.
\end{equation}
We can similarly obtain the bound 
\begin{equation}
|L_{N_1}(x) - L_{\beta N_{0}'}(x)| < C\kappa^{1/2}
\end{equation}
for $\beta = c \alpha$ such that $\beta N_{0}' < N_1 < (\beta + 1) N_{0}'$ (away from a set of measure at most $\kappa).$ Finally,
$$|L_{\beta N_{0}'}(x) - L_{\alpha N_{0}'}(x)| < C\kappa^{1/2}$$
by \eqref{eq:96}.
Combining all of these with the triangle inequality, and Lemma \ref{Lem:UniformL2} we obtain our desired result.
\end{proof}



Our next step is to extend the conclusion of Lemma \ref{Lem:Liouv} to allow for larger scales $N_1.$ This is most directly achieved by iterating Lemma \ref{Lem:Liouv}: informally, we first apply the result with initial scale $N_0$ to obtain an estimate between scales $N_0$ and $N_1;$ then apply the result with initial scale $N_1$ to obtain an estimate between scales $N_1$ and new scale $N_2;$ repeating this iteration until we can no longer choose a suitable new initial scale. Unfortunately, as the result is stated now, this results in an error given by a multiple of $\kappa,$ and that multiple could, potentially, be large enough to overcome the desired $\kappa$ error. Therefore, it is necessary for us to rewrite the conclusion of Lemma \ref{Lem:Liouv} in a way which expresses the error in terms of $N_0$ and $N_1.$ This is achieved in the following.

\begin{mylemma}[Liouville Frequencies]\label{LiouvilleLemma}
Let $N_0, q_0 \in \N$ and $c = C^{-1},$ with $C$ as in Lemma \ref{Lem:Liouv}. Suppose $N_0 > q_0$ and $\norm{q_0\omega} < \rho^4 \frac{q_0}{N_0}.$ Then for $N_0 | N_1$ with $N_1 < \rho^2 \sqrt{\frac{N_0q_0}{\norm{q_0\omega}}},$
\begin{equation}
\left|L_{N_1} - L_{N_0}\right| < C'\left(\left(\frac{q_0}{N_0}\right)^{c/2} + \left(\frac{N_1^2\norm{q_0\omega}}{N_0q_0\rho^4}\right)^{c/2}\right).
\end{equation} Here $C' = C'(A)$ is a constant uniform in a neighborhood of the cocycle $A.$
\end{mylemma}
\begin{proof}
Fix $N_0, q_0,$ and $\rho$ so that $N_0 > q_0$ and $\norm{q_0\omega} < \rho^4 \frac{q_0}{N_0}.$ Then let $1 > \kappa' > 0$ be the smallest possible $\kappa$ such that Lemma \ref{Lem:Liouv} is applicable. That is, $\kappa'$ is the smallest such number for which \eqref{eq:LiouvCondition}, \eqref{eq:N0LargeCondition}, and \eqref{eq:N1SmallCondition} all hold. Then $|L_{N_0} - L_{N_1}| < C'\kappa'.$ Moreover, equality must hold for one of \eqref{eq:LiouvCondition}, \eqref{eq:N0LargeCondition}, or \eqref{eq:N1SmallCondition}. 

If equality holds in \eqref{eq:LiouvCondition}, then $$\kappa' = \left(\frac{q_0}{N_0}\right)^c$$ and our conclusion follows. 

If, on the other hand, equality does not hold in \eqref{eq:LiouvCondition}, then equality must hold in \eqref{eq:N1SmallCondition}, since $N_1 > N_0.$ Thus 
$$\kappa' = \left(\frac{N_1^2\norm{q_0\omega}}{N_0q_0\rho^4}\right)^c$$
and our conclusion follows.
\end{proof}

It is now possible to iterate Lemma \ref{Lem:Liouv} as we described to allow for larger scales $N_1.$

\begin{mythm}\label{Thm:Liouv}
Suppose $N_0\kappa^C > q_0$ and $N_0 \norm{q_0\omega} < \kappa^C \rho^4 q_0.$ Then 
\begin{equation}
\left|L_{N'} - L_{N_0}\right| < C' \kappa^{1/6}
\end{equation}
for all $N'$ such that $N_0 | N', N' = 2^jN_0$ for some $j \geq 0,$ and 
\begin{equation}
N' < \kappa^C \rho^4 \frac{q_0}{\norm{q_0\omega}}.
\end{equation}
\end{mythm}

\begin{proof}

Fix $N' < \kappa^C \rho^4 \frac{q_0}{\norm{q_0\omega}}.$ We will construct a sequence of scales, $N_s,$ inductively. Starting with $N_0$ such that $q_0 < \kappa^C N_0$ and $\norm{q_0\omega} < \kappa^C \rho^4 \frac{q_0}{N_0},$ we define
\begin{equation}
N_s \sim N_{s-1}^{2/3} \left(\frac{q_0\rho^4}{\norm{q_0\omega}}\right)^{1/3}
\end{equation}
where $\sim$ here indicates that we take 
any
value no larger than the right hand side such that $N_{s-1} | N_s$ and $N_{s -1} \ne N_s.$ This last condition is possible because, by an inductive argument, $N_{s-1}^{2/3}\left(\frac{q_0\rho^4}{\norm{q_0\omega}}\right)^{1/3} > N_{s-1}\kappa^{-C} \geq 2N_{s-1}.$ Observe that this last condition ensures that $N_s$ is an increasing sequence, and thus there is some $s_0 \geq 1$ such that 
$$N_{s_0} < \kappa^C \rho^4 \frac{q_0}{\norm{q_0\omega}}$$
and
$$N_{s_0 + 1} > \kappa^C \rho^4 \frac{q_0}{\norm{q_0\omega}}.$$
Moreover, we have, for $0 \leq s \leq s_0,$
\begin{align}
N_s &\leq N_{s - 1}^{2/3}\left(\frac{q_0\rho^4}{\norm{q_0\omega}}\right)^{1/3}\\
&= N_{s - 1}^{1/2}N_{s - 1}^{1/6}\left(\frac{q_0\rho^4}{\norm{q_0\omega}}\right)^{1/3}\\
&\leq N_{s - 1}^{1/2} \kappa^{C/6}\left(\frac{q_0\rho^4}{\norm{q_0\omega}}\right)^{1/6}\left(\frac{q_0\rho^4}{\norm{q_0\omega}}\right)^{1/3}\\
&= \kappa^{C/6}\rho^2\left(\frac{q_0N_{s - 1}}{\norm{q_0\omega}}\right)^{1/2}\\
&\leq \rho^2\left(\frac{q_0N_{s - 1}}{\norm{q_0\omega}}\right)^{1/2}.
\end{align}
Thus Lemma \ref{LiouvilleLemma} is applicable at scales $N_s$ and $N_{s - 1}$ when $0 \leq s \leq s_0.$
Applying Lemma \ref{LiouvilleLemma}, we obtain, for $s \leq s_0,$
\begin{align}|L_{N_s} - L_{N_{s - 1}}| &< C' \left(\left(\frac{q_0}{N_{s-1}}\right)^{c/2} + \left(\frac{N_s^2\norm{q_0\omega}}{N_{s-1}q_0\rho^4}\right)^{c/2}\right)\\
&\leq C' \left(\left(\frac{q_0}{N_{s-1}}\right)^{c/2} + \left(\frac{N_{s-1}\norm{q_0\omega}}{q_0\rho^4}\right)^{c/6}\right).
\end{align}

Now we consider any $N' < \kappa^C\rho^4\frac{q_0}{\norm{q_0\omega}}.$ There are two possibilities: $(i)$ $N_s < N' \leq N_{s+1}, s < s_0,$ or $(ii)$ $N' > N_{s_0}.$ 

In the first case, we may simply redefine the appropriate $N_{s + 1} = N'.$ Then, we observe that $2^{j - 1}N_0 \leq N_{j - 1} \leq 2^{-(s_0 + 1 - j)}\kappa^C\rho^4\frac{q_0}{\norm{q_0\omega}},$ by construction.  This yields 
\begin{align}
|L_{N_0} - L_{N'}| &\leq \sum_{j = 1}^s |L_{N_{j-1}} - L_{N_j}|\\
&\leq \sum_{j = 1}^s C' \left(\left(\frac{q_0}{N_{j-1}}\right)^{c/2} + \left(\frac{N_{j-1}\norm{q_0\omega}}{q_0\rho^4}\right)^{c/6}\right)\\
&\leq \sum_{j = 1}^s C' \left(\left(\frac{q_0}{N_{0}}2^{1 - j}\right)^{c/2} + \left(2^{-(s_0 + 1 - j)} \kappa^C\right)^{c/6}\right)\\
&\leq C'' \left(\left(\frac{q_0}{N_0} \right)^{c/2} + \kappa^{1/6}\right)\\
&\leq C''(\kappa^{1/2} + \kappa^{1/6})\\
&\leq C'' \kappa^{1/6}.
\end{align}

For case $(ii),$ we may repeat the same argument as above, plus the observation that
\begin{align}
N' &< N_{s_0+1} \\
&\sim N_{s_0}^{2/3} \left(\frac{q_0\rho^4}{\norm{q_0\omega}}\right)^{1/3}\\
&< \kappa^{C/6} N_{s_0}^{1/2} (q_0\rho^4/\norm{q_0\omega})^{1/2}.
\end{align}
Now, we may apply lemma \ref{Lem:Liouv} with $N_0 = N_{s_0}$ and $N_1 = N'$ and $\kappa$ replaced by $\kappa^{1/3}$ to obtain the estimate
$$|L_{N_{s_0}} - L_{N'}| < C' \kappa^{1/6}.$$
This completes our proof.
\end{proof}

Note that this immediately implies continuity of $L(A,\omega)$ in $A$ whenever the components of $\omega,$ are rational.


\section{Comparing Lyapunov exponents at different scales: mixed Liouville-Diophantine frequencies} \label{Section:Mixed}


In this section, we turn out attention to obtaining estimates of the form $|L_{N_0}(A,\omega) - L_{N_1}(A,\omega)| < C\kappa$ for those $\omega$ which are not necessarily Liouville. 

\begin{mythm}\label{Thm:Mixed}
Assume $x = (x_1, x_2) \in \T^{d_1} \times \T^{d_2} = \T^d.$ Suppose that $\omega = (\omega_1, \omega_2) \in \T^{d_1} \times \T^{d_2}$ is such that there is $\delta > 0, 0 < K_0 \in \Z$  and $0 < q_0 \in \Z$ such that 
\begin{equation}\label{eq:Mixed:hyp1}
\norm{q_0\omega_1} < \kappa^C \rho^3 q_0/N_0
\end{equation}
and
\begin{equation}
\norm{k \cdot \omega_2} > \delta
\end{equation}
for  all $k \in \Z^{d_2}, 0 < |k| < K_0,$ where
\begin{align}
K_0 &> (\rho^{1 + c}\kappa)^{-C}q_0\label{eq:Mixed:hyp2}\\
N_0 &> \kappa^{-C}\delta^{-1}K_0.\label{eq:Mixed:hyp3}
\end{align}
Then 
\begin{equation}
|L_N - L_{N_0}| < C'\kappa^{1/6}
\end{equation}
when $N_0 | N, N = 2^jN_0$ for some $j \geq 0,$ and 
\begin{equation}
N < \min\set{\kappa^C \rho^3 \frac{q_0}{\norm{q_0\omega_1}}, N_0 e^{\left(\frac{K_0}{q_0}\right)^c}}.
\end{equation} 
Here $C$ is an absolute constant defined in the proof and $c = C^{-1}$ above.
\end{mythm}

We would like to make a note about our general approach to the proof, before we present the full details. As we remarked after Lemma \ref{Lem:Liouv}, the argument we used to establish Lemma \ref{Lem:Liouv} applies once we restrict our attention to a suitable set $F.$ Here, we will restrict to a suitable set $F$ defined analogously, but we cannot just appeal to Lemma \ref{Lem:CDT} to show $\T^d \backslash F$ has small measure, since Lemma \ref{Lem:CDT} requires a small change in $x,$ and we assume a Diophantine condition (corresponding to a large change in $x$) for part of $\omega.$ Therefore, the main difficulty we face here is showing $\T^d \backslash F$ has small measure.  
The idea we present here is to use Lemma \ref{Lem:CDT} in the variables where we have suitable smallness of the frequency (namely $x_1$ with $\omega_1$) and use Theorem \ref{Thm:LDT} in the variables where we have a suitable Diophantine condition. Together, this will yield smallness of $\T^d \backslash F.$


\begin{proof}
As in Lemma \ref{Lem:Liouv}, we will consider those $N_0$ of the form $a2^{b}q_0,$ for suitably large $b.$ The case of general $N_0$ will then follow in the same way as before.

Consider the set $F$ consisting of points $x = (x_1,x_2) \in \T^{d_1} \times \T^{d_2} = \T^d$ such that 
$$\left|L_N(x) - L_N(x + jq_0\omega)\right| < \kappa$$
for all $2^{-s_0}N_0 \leq N \leq N_0,$ where $s_0 = - C_1\ln\kappa,$ 
$C_1 > 4,$ with $C>0$ 
such that $C > 6(3C_1 + 3) + 1,$ $N$ is of the form $2^{-s}N_0,$ and $j \leq N_1/q_0$ is such that $N_0/2^sq_0$ divides $j$ for some $s \leq s_0.$  We are in precisely the setting of Lemma \ref{Lem:BaseEstAP}, so we conclude that
\begin{equation}
|\int_F L_{N_1}(x) - L_{N_0}(x) dx| < C''\kappa.
\end{equation}
It now suffices to understand $|\T^d \backslash F|.$


We will begin by restricting our attention to the set 
$$H_N = \set{x\in \T^d: |L_N(x) - L_N(x + jq_0\omega)|< C(A) |j| q_0 N^{-1/2}, 1 \leq j \leq \kappa N/q_0 },$$
which, by Lemma \ref{Lem:CDTVar}, has small complement
$$|\T^d\backslash H_N| < \kappa e^{-N^{1/3}}.$$ 

As in Lemma \ref{Lem:BaseEstAP}, we observe that Lemma \ref{Lem:CDT} applies to $L_N(x_1,x_2)$ in the first variable, which yields
\begin{align*}
L_N(x + jq_0\omega) &= L_N(x_1 + j q_0 \omega_1, x_2 + j q_0 \omega_2)\\
&= L_N(x_1, x_2 + jq_0\omega_2) + O(\kappa),
\end{align*}
away from a set $G_{N,j}\subset \T^d$ of small measure:
$$|G_{N,j}| < C' \kappa^{-3}\rho^{-3} |j|\norm{q_0\omega_1}.$$
Hence
\begin{equation}
L_N(x + jq_0\omega) = L_N(x_1, x_2 + jq_0\omega_2) + O(\kappa) + g_{N,j}(x),
\end{equation}
where $g_{N,j}(x)$ is a function such that $\norm{g_{N,j}}_{L^1(\T^d)} \leq C' \kappa^{-3}\rho^{-3} |j|\norm{q_0\omega_1}.$ In particular, $g_{N,j}(x)$ is the restriction of $L_N(x)$ to the set $G_{N,j}.$
Moreover, for $x \in H_N \cap \bigcup_{j = 0}^{R - 1}(\T^d\cap G_{N,j})$ and $R < \kappa \frac{\sqrt N}{q_0},$
\begin{align}
L_N(x) &= \frac1R \sum_{j = 0}^{R - 1} L_N(x + jq_0\omega) + C' R q_0/\sqrt{N}\\
&= \frac1R \sum_{j = 0}^{R - 1} L_N(x_1, x_2 + jq_0\omega_2) + O(\kappa)\label{eq:LotsofO}
\end{align}
and thus
\begin{equation}
L_N(x) = \frac1R \sum_{j = 0}^{R - 1} L_N(x_1, x_2 + jq_0\omega_2) + O(\kappa) + g_N(x),
\end{equation}
where $g_N(x) = \frac1R \sum_{j = 0}^{R - 1} g_{N,j}(x)$ satisfies $\norm{g_{N}}_{L^1(\T^d)} \leq C' \kappa^{-3}\rho^{-3} R\norm{q_0\omega_1}.$
Note that the $O(\kappa)$ term is no larger than $2\kappa.$ 

Now, by defining $\omega_2' = q_0\omega_2,$ we have $\norm{k \cdot \omega_2'} > \delta$ for all $0 < |k| < K_0/q_0.$ Thus we may apply Lemma \ref{Lem:LDTver1} to $L_N(x_1,x_2)$ in the variable $x_2$ to obtain
\begin{align*}
\bigg|\bigg\{(x_1,x_2)\in \T^d: \bigg|\frac{1}{R} \sum_{j = 0}^{R - 1} L_N(x_1, x_2 + jq_0\omega_2) &- \int L_N(x_1,x_2) dx_2\bigg| > \rho^{-1}(K_0/q_0)^c\bigg\} \bigg|\\
&< e^{-\rho^{1 + c}(K_0/q_0)^c}.
\end{align*}
Denote by $\Gamma_N$ the set on the left hand side.
We have, $\rho^{-1}(K_0/q_0)^{-c} < \kappa,$ so for $(x_1,x_2) \not\in \Gamma_N,$
$$\frac{1}{R} \sum_{j = 0}^{R - 1} L_N(x_1, x_2 + jq_0\omega_2) = \int L_n(x_1, x_2) dx_2 + O(\kappa).$$ 
On the other hand, the integral of $\frac{1}{R} \sum_{j = 0}^{R - 1} L_N(x_1, x_2 + jq_0\omega_2) $ over $\Gamma_N$ obeys
$$\int_{\Gamma_N} \left|\frac{1}{R} \sum_{j = 0}^{R - 1} L_N(x_1, x_2 + jq_0\omega_2)  \right|dx_2 \leq C' e^{-\frac12\rho^{1 + c}(K_0/q_0)^c},$$ 
by Lemma \ref{Lem:UniformL2}.  
Hence
\begin{equation}
\frac{1}{R} \sum_{j = 0}^{R - 1} L_N(x_1, x_2 + jq_0\omega_2)  = \int L_n(x_1, x_2) dx_2 + O(\kappa) + \gamma_N(x),
\end{equation}
where $\gamma_N(x)$ is the restriction of $L_N(x)$ to $\Gamma_N$ and satisfies $\norm{\gamma_N}_{L^1(\T^d)} < e^{-\frac12\rho^{1 + c}(K_0/q_0)^c}.$ Combining this with \eqref{eq:LotsofO}, we have
\begin{equation}
L_N(x) = \int L_N(x_1, x_2) dx_2 + O(\kappa) + g_N(x) + \gamma_N(x).
\end{equation}

We may apply Lemma \ref{Lem:CDT} again in $x_1$ to $\int L_N(x_1,x_2)dx_2$ to obtain
\begin{equation}
\left|\int L_N(x_1, x_2)dx_2 - \int L_N(x_1 + jq_0\omega_1, x_2 + jq_0\omega_2) dx_2 \right| = O(\kappa) + h_{N,j}(x),
\end{equation}
where $h_{N,j}(x)$ is a suitable restriction of $L_N(x)$ and satisfies $\norm{h_{N,j}(x)}_{L^1} < \kappa^{-3}\rho^{-3}|j|\norm{q_0\omega_1}.$

Altogether, we have
\begin{equation}
\left|L_N(x) - L_N(x + jq_0\omega)\right| = O(\kappa) + g_N(x) + \gamma_N(x) + h_{N,j}(x).
\end{equation}

It now follows that
$$|\T^d \backslash F| < \kappa^{-C'}\rho^{-3}\frac{N_1^2\norm{q_0\omega_1}}{q_0N_0} + \kappa^{-C'}\frac{N_1}{N_0}e^{-\rho^{1 + c}\left(\frac{K_0}{q_0}\right)^c}.$$
Here $c$ is a sufficiently small constant which is anything smaller than both the constant from Theorem \ref{Thm:LDT} and $1/C,$ and $C'$ is such that $C/3 > C'.$
In particular, this computation follows by observing that $\T^d\backslash F$ is contained in the union of the supports of $g_N, \gamma_N,$ and $h_{N,j}$ over all relevant $N$ and $j.$ 

It follows that
\begin{equation}
\left|L_{N_1} - L_{N_0}\right| < C''\kappa + \kappa^{-C'}\rho^{-3}\frac{N_1^2\norm{q_0\omega_1}}{q_0N_0} + \kappa^{-C'}\frac{N_1}{N_0}e^{-\rho^{1 + c}\left(\frac{K_0}{q_0}\right)^c}.
\end{equation}
Now, if $N_1$ is such that
$$\rho^{-3}\frac{N_1^2\norm{q_0\omega_1}}{q_0N_0} + \frac{N_1}{N_0}e^{-\rho^{1 + c}\left(\frac{K_0}{q_0}\right)^c} < \kappa^{1 + C'},$$ 
then we have our desired bound. Moreover, if
$$\frac{N_1}{N_0} \leq \min\set{\left(\frac{\rho^3q_0}{N_0\norm{q_0\omega_1}}\right)^{1/3}, e^{\frac12 \rho^{1 + c} (K_0/q_0)^c}},$$
then by direct computation (using \eqref{eq:Mixed:hyp1}) such a bound is satisfied. 

We can now use the exact same argument as was used in Theorem \ref{Thm:Liouv} to extend our allowable length scales $N_1$ to the desired range, since the scale $$N_0 \left(\frac{\rho^3q_0}{N_0\norm{q_0\omega_1}}\right)^{1/3}$$ is precisely the intermediate scale we used to extend Lemma \ref{Lem:Liouv} to Theorem \ref{Thm:Liouv}.
\end{proof}
 

\section{Continuity of Lyapunov exponents}\label{Section:ContLemma}


The main technical lemma which we will establish at the end of this section is the following.
\begin{mylemma}\label{Lem:FinalLemma}
Assume $x = (x_1, x_2) \in \T^{d_1} \times \T^{d_2}, \omega = (\omega_1, \omega_2) \in \T^{d_1} \times \T^{d_2}$ such that
$$\norm{q_0 \omega_1} = 0$$
and
$$\norm{k\cdot\omega_2} > \delta \quad \text{for } k\in \Z^{d_2}, 0 < |k| \leq K,$$
where
$$q_0 < K^{1/10}.$$
Then
$$|L_N - L| < \kappa$$
for all $N > K^2/\delta.$
\end{mylemma}

This lemma will be sufficient to obtain continuity due to the following observation. Every $\omega = (\omega_1,...,\omega_d) \in \T^d$ falls into one of three categories: $(i)$ $k \cdot \omega \ne 0$ for every $k \in \Z^d\backslash \{0\},$ 
$(ii)$ $\omega_j \in \Q$ for some $j,$ or $(iii)$ $\omega_j$ are all irrational but rationally dependent. The above lemma clearly applies to the $\omega$ in the first two of these possibilities. The last possibility can be ``transformed'' into the second possibility after an application of an appropriate toral automorphism (i.e. a suitable change of variables), $B \in SL(d,\Z).$ Indeed, for any $B \in SL(d,\Z),$ we have:
\begin{align}
L_N(A,\omega) &= \frac 1 N \int \ln\norm{\prod_{j = N - 1}^0 A(x + j\omega)} dx \\
&= \frac 1 N \int \ln\norm{\prod_{j = N - 1}^0 (A\circ B)(x + jB^{-1}\omega)} dx \\
&= L_N(A\circ B, B^{-1}\omega).
\end{align}
Since $(A\circ B, B^{-1}\omega)$ is still an analytic quasiperiodic cocycle, all of the results from the previous sections apply. The key idea is to now find an appropriate such $B$ so that $B^{-1}$ rearranges the components of $\omega$ into two pieces: one piece consisting of irrational and rationally independent components, and another piece consisting of rationals and rationally dependent components.

The main step towards achieving this is the following theorem.

\begin{mythm}\label{Thm:MainStep}
Let $(\omega_1, \omega_2) \in \T^{d_1}\times \T^{d_2}, d_1 + d_2 = d.$ Let $1 \geq \delta_0 > 0, \epsilon_0 \geq 0,$ and suppose $q_0 \in \N$ and $K_0 \in \N,$ with $q_0 < K_0^{1/10},$ satisfies
\begin{equation}
\norm{q_0\omega_1} \leq \epsilon_0,
\end{equation}
\begin{equation}
\norm{k\omega_2} \geq \delta_0; \quad 0 < |k| \leq K_0.
\end{equation}
Furthermore, suppose $N_0$ is such that
\begin{equation}
\frac12 K_0^2 \delta_0^{-1} \leq N_0 < \epsilon_0^{-1}K_0^{-1}.
\end{equation}
Finally, suppose $\rho > K_0^{-c}.$ Then for $N_0 | N_1, N_1 = 2^jN_0$ for some $j \geq 0,$ and $N_1 \leq \epsilon^{-1}K_0^{-1},$ we have
\begin{equation}\label{eq:FinalConclusion}
|L_{N_0} - L_{N_1}| < K_0^{-c}.
\end{equation}
\end{mythm}

We will prove this by induction, but before we present the proof, we will provide a lemma which reduces the above result to proving analogous bounds on shorter length scales.

\begin{mylemma}\label{Lem:MainStep}
In addition to the assumptions of Theorem \ref{Thm:MainStep},
suppose, moreover, that 
\begin{equation}
|L_{N_0} - L_{N'}| < K_0^c
\end{equation}
for $N_0 | N', N' = 2^j N_0$ for any $j \geq 0$ such that 
\begin{equation}\label{eq:RestrN1Cond}
N' \leq \min\set{\epsilon_0^{-1}K_0^{-1}, K_0^{40}\left(\min_{0 < |k| \leq K_0^{20}} \norm{k\cdot\omega_2}\right)^{-1} + N_0}.
\end{equation}
Then \eqref{eq:FinalConclusion} holds for $N_0 | N_1, N_1 = 2^jN_0$ for any $j \geq 0$ such that $N_1 \leq \epsilon^{-1}K_0^{-1}.$
\end{mylemma}

\begin{proof}
Suppose \eqref{eq:FinalConclusion} holds for $N'$ as above. We will construct a sequence of length scales, $N_{0,s},$ and iterate the conclusion in a way that mirrors our proof of Theorem \ref{Thm:Liouv}.

Note that, if 
$$K_0^{40}\left(\min_{0 \leq |k| \leq K_0^{20}} \norm{k\cdot\omega_2}\right)^{-1} + N_0 > \epsilon_0^{-1}K_0^{-1},$$ 
then there is nothing to prove, so we will assume
$$K_0^{40}\left(\min_{0 \leq |k| \leq K_0^{20}} \norm{k\cdot\omega}\right)^{-1} + N_0 \leq \epsilon_0^{-1}K_0^{-1}$$
Set $K_1 = K_0^{20}, \delta_1 = \min_{0 < |k| \leq K_1} \norm{k\cdot\omega_2}.$ Then define
\begin{equation}
N_{0,1} \sim \frac{K_1^2}{\delta_1} + N_0,
\end{equation}
where $\sim$ here means take 
the largest 
multiple of $N_0$ no larger than the right hand side of the form $N_{0,1} = 2^j N_0.$ 
By our assumptions, $N_{0,1}$ satisfies \eqref{eq:RestrN1Cond}, and thus
$$|L_{N_0} - L_{N_{0,1}}| < K_0^{-c}.$$
Moreover, suppose $N_{0,1} < \epsilon_0^{-1}K_1^{-1}.$ We will deal with the case where $N_{0,1} \geq \epsilon_0^{-1}K_1^{-1}$ at the end of the proof using Theorem \ref{Thm:Mixed}.

Now we observe that we may replace $\delta_0, K_0,$ and $N_0$ in the hypothesis of our lemma with $\delta_1, K_1,$ and $N_{0,1},$ respectively. It then follows, by our assumptions, that 
$$|L_{N'} - L_{N_{0,1}}| < K_1^{-c}$$
for all $N'$ such that $N_{0,1} | N'$ and 
$$N' \leq \min\set{\epsilon_0^{-1}K_1^{-1}, K_1^{40}\left(\min_{0 < |k| \leq K_1^{20}} \norm{k\cdot\omega_2}\right)^{-1} + N_{0,1}}.
$$

At this point, two possibilities arise. First, suppose
\begin{equation} \label{eq:InductiveAssump1}
K_1^{40}\left(\min_{0 < |k| \leq K_1^{20}} \norm{k\cdot\omega_2}\right)^{-1} + N_{0,1} \leq \epsilon_0^{-1}K_1^{-20}.
\end{equation}
We will see later that the case where \eqref{eq:InductiveAssump1} fails may be dealt with via Theorem \ref{Thm:Mixed}. In fact, this assumption is analogous to the assumption $N_{0,1} < \epsilon_0^{-1}K_1^{-1}$ we made above. 
Set $K_2 = K_1^{20}, \delta_2 = \min_{0 < |k| \leq K_2}\norm{k\cdot\omega_2}.$ Then define
\begin{equation}
N_{0,2} \sim \frac{K_2^2}{\delta_2} + N_{0,1},
\end{equation}
where $\sim$ here means take 
the largest 
multiple of $N_{0,1}$ of the form $N_{0,2} = 2^j N_{0,1}$ which is no larger than the right hand side. By our assumptions, $N_{0,2}$ satisfies \eqref{eq:RestrN1Cond}, and thus
$$|L_{N_{0,1}} - L_{N_{0,2}}| < K_1^{-c}$$
and
$$|L_{N'} - L_{N_{0,2}}| < K_2^{-c}$$
for all $N'$ such that $N_{0,2} | N'$ and 
$$N' \leq \min\set{\epsilon_0^{-1}K_2^{-1}, K_2^{40}\left(\min_{0 < |k| \leq K_2^{20}} \norm{k\cdot\omega_2}\right)^{-1} + N_{0,2}}.
$$

We continue in this way to define $K_s = K_{s-1}^{20}, \delta_s = \min_{0 < |k| \leq K_s}\norm{k\cdot\omega_2},$ and $N_{0,s}\sim \frac{K_s^2}{\delta_s} + N_{0,s-1}$ for $s \leq s_0,$ where $s_0$ is the first index for which
$$K_{s_0}^{40}\left(\min_{0 < |k| \leq K_{s_0}^{20}} \norm{k\cdot\omega_2}\right)^{-1} + N_{0,s_0} > \epsilon_0^{-1}K_{s_0}^{-1}.$$
By construction, for all $s \leq s_0,$
\begin{equation}
|L_{N_{0,s}} - L_{N_{0,s-1}}| < K_{s-1}^{-c}
\end{equation}
and for $s < s_0,$
\begin{equation}
|L_{N_{0,s}} - L_{N'}| < K_{s-1}^{-c},
\end{equation}
for all $N' = 2^j N_{0,s},$ where 
$$N' \leq \min\set{\epsilon_0^{-1}K_s^{-1}, K_{s+1}^{2}\delta_{s+1} + N_{0,s}}.$$

Now set $N_{0,s_0}' \sim \max\set{N_{0,s_0-1}, \epsilon_0^{-1}K_{s_0}^{-1}},$ where, again, $\sim$ denotes taking the largest value no larger than the right hand side such that $N_{0,s_0-1}|N_{0,s_0}'.$ Thus, by construction,
$$|L_{N_{0,s_0-1}} - L_{N_{0,s_0}}'| < K_{s_0-1}^{-c}.$$

At this point, we may apply Theorem \ref{Thm:Mixed} to the scale $N_{0,s_0}',$ with $K_{s_0-1}$ and $\delta_{s_0-1}.$ Indeed,
\begin{align}
N_{0,s_0}'\epsilon_0 &\leq \max\set{N_{0,s_0-1}\epsilon_0, K_{s_0}^{-1}}\\
&\leq \max\set{K_{s_0-1}^{-1}, K_{s_0}^{-1}}\\
&\leq K_{s_0-1}^{-1}\\
&\leq K_0^{-1}\\
&\leq \kappa^C\rho^4q_0,
\end{align}
where $\kappa = K_0^{-c},$ and $C$ is chosen appropriately large, depending on $c$ we used in the hypothesis of this theorem. Thus \eqref{eq:Mixed:hyp1} holds.
We also have
\begin{align}
K_0 > \rho^{-1/C}\kappa^{-1/C}q_0,
\end{align}
so \eqref{eq:Mixed:hyp2} holds.
Finally, 
\begin{align}
N_{0,s_0}' &\geq N_{0,s_0-1}\\
&\geq K_{s_0-1}^2 \delta_{s_0-1}^{-1}\\
&\geq K_{s_0-1} \delta_{s_0-1}^{-1} \kappa^{-C},
\end{align}
and thus \eqref{eq:Mixed:hyp3} holds. It follows that Theorem \ref{Thm:Mixed} is applicable with our chosen parameters, and we obtain
\begin{equation}
|L_{N_{s_0}} - L_{N'}| < K_0^{-c}
\end{equation}
for all $N'$ of the form $2^j N_{0,s_0}$ such that
$$N' \leq \min\set{\epsilon_0^{-1}K_0^{-1}, \epsilon_0^{-1}}.$$
This now includes all $N' \leq \epsilon_0^{-1} K_0^{-1},$ as desired.

Now we fix any $N' = 2^j N_0.$ Two possibilities arise. First, suppose $N_{0,s} < N_1 \leq N_{0,s+1}$ for some $0 \leq s \leq s_0.$ By construction of $N_{0,s+1},$ we have
$$|L_{N_{0,s}} - L_{N'}| < K_s^{-c},$$
and thus
\begin{align}
|L_{N_0} - L_{N'}| &< |L_{N_0} - L_{N_{0,1}}| + \sum_{j = 1}^{s-1}|L_{N_{0,j}} - L_{N_{0,j+1}}| \\
\quad &+ |L_{N_{0,s}} - L_{N'}|\\
&\leq K_0^{-c} + \sum_{j = 1}^{s-1} K_j^{-c} + K_s^{-c}\\
&= K_0^{-c} + \sum_{j = 1}^{s-1} K_0^{-c20^j} + K_0^{-c20^s}\\
&\leq C K_0^{-c},
\end{align}
where $C$ is an absolute constant.

Now, suppose $N' > N_{0,s_0}.$ As we observed above,
$$|L_{N_{s_0}} - L_{N'}| < K_0^{-c},$$
and the same argument as the previous case yields our desired result.

\end{proof}

At this point, we provide a lemma guaranteeing the existence of the special toral automorphisms we discussed at the beginning of this section.

\begin{mylemma}[\cite{BourgainContinuity} Lemma 4.41] \label{Lem:CoV}Assume $k = (k_1,...,k_d) \in \Z^d$ and $\gcd(k_1,...,k_d) = 1.$ Then there is $A \in SL_d(\Z)$ satisfying
\begin{equation}
A_{1j} = k_j;\quad 1 \leq j \leq d,
\end{equation}
\begin{equation}
|A_{ij}| \leq |k| = \max |k_l|;\quad 1 \leq i,j \leq d.
\end{equation}
\end{mylemma}

We will now use this in an induction scheme to prove Theorem \ref{Thm:MainStep}.

\begin{proof}[proof of Theorem \ref{Thm:MainStep}]
By Lemma \ref{Lem:MainStep}, it suffices to prove our result for $N_1$ which satisfies \eqref{eq:RestrN1Cond}. We will prove this by induction on $d_2.$ 

First, we consider the base case, $d_2 = 0.$ Set $\kappa = K_0^{-c},$ with $c$ a suitably small absolute constant. Set $C = c^{-1}.$ Then 
$$N_0 \kappa^{C} = N_0 K_0^{-1} > K_0 > q_0^{10} > q_0.$$
Moreover, 
$$N_0 \norm{q_0\omega_1} \leq N_0 \epsilon_0 < K_0^{-1} < \kappa^{C} < \kappa^{C} \rho^4 q_0.$$
Thus Theorem \ref{Thm:Liouv} is applicable, and we obtain
$$|L_{N_0} - L_{N_1}| < K_0^{-c/3}$$
for $$N_1 < K_0^{-1/2} \rho^4 q_0/\norm{q_0\omega_1}.$$ Since $\epsilon^{-1}K_0^{-1} < K_0^{-1/2} \rho^4 q_0 / \norm{q_0\omega_1},$ our conclusion follows.

Now we turn our attention to the inductive step. Assume \eqref{eq:FinalConclusion} holds for $\omega_2 \in \T^{d_2 - 1}$ and $N_1$ satisfying \eqref{eq:RestrN1Cond}. We will show that \eqref{eq:FinalConclusion} is true for $\omega_2 \in \T^{d_2}.$ 

Fix any suitable $\omega_2 \in \T^{d_2}.$ Set $K_1 = K_0^{20}$ and $\delta_1 = \min_{0 < |k| \leq K_1} \norm{k \cdot \omega_2}.$ Note that, if $N_0 \geq K_1^2\delta_1^{-1},$ then any $N_1$ satisfying \eqref{eq:RestrN1Cond} also satisfies $N_1 < 2N_0$ and our result is vacuously true. Thus, we will assume $N_0 < K_1^2\delta_1^{-1}.$

Now, by the definition of $\delta_1,$ there is some $q' \in \Z^{d_2},$ with $0 < |q'| \leq K_1,$ such that the minimum is achieved. That is, we can find $q' = (q_{1}',\dots,q_{d_2}')\in \Z^{d_2},$ with $0 < |q'| \leq K_1,$ such that $\norm{q'\omega_2} = \delta_1.$ Write $q' = q_1\cdot n_1,$ where $q_1 \in \Z\backslash \set{0}, |q_1| \leq K_1,$ and $n_1 = (n_{11},...,n_{1d_2})$ is such that $\gcd(n_{11},...,n_{1d_2}) = 1.$ 

At this point, our goal is to perform a suitable change of variables to reduce our situation to that covered in our induction hypothesis.
We apply Lemma \ref{Lem:CoV} with $k = n_1$ to construct $B \in SL_d(\Z)$ with entries bounded by $K_1$ such that 
\begin{equation}
B(\omega_1,\omega_2) = \omega' = (\omega_1', \omega_2')\in \T^{d_1+1}\times \T^{d_2-1},
\end{equation} 
where 
\begin{equation}
\norm{q_1'\omega_1'} = \delta_1.
\end{equation} 
We may also assume that $B$ fixes the first $d_1$ components of $\omega.$ Moreover, by Cramer's rule, the entries of $B^{-1}$ are all bounded by $K_1^{d-1}.$ 

Now, since $B\in SL_d(\Z),$ we have, for every $N,\omega,$
\begin{align}
L_N(A,\omega) &= \frac 1 N \int_{\T^d} \ln\norm{\prod_{j = N}^1 A(x + j\omega)} dx\\
&= \frac 1 N \int_{\T^d} \ln\norm{\prod_{j = N}^1 AB^{-1}(B x + j B \omega)} dx\\
&= \frac 1 N \int_{\T^d} \ln\norm{\prod_{j = N}^1 AB^{-1}(x + j B \omega)} dx\\
&= L_N(AB^{-1}, B\omega).
\end{align}
Here, the second to last equality is simply a change of variables, $Bx \mapsto u,$ and we take advantage of the fact that $B \in SL_d(\Z).$ Thus, to understand the Lyapunov exponent for the cocycle $(A,\omega),$ we may study the related cocycle $(AB^{-1}, B\omega) = (AB^{-1}, \omega').$ We now want to show that this new cocycle satisfies the induction hypothesis.

Recall that $L_N(A, \omega)$ has a plurisubharmonic extension to $|\Im z_j| < \rho = \rho_0$ which satisfies the conditions necessary to establish the results in the previous sections. It follows that $L_N(AB^{-1}, \omega')$ also has a plurisubharmonci extension to $|\Im z_j| < \norm{B^{-1}}\rho_0$ for which the results of the previous section apply. Since $\norm{B^{-1}} < K_1^{d-1},$ the extension can certainly be restricted to
$$|\Im z_j | < \rho_1 = K_1^{1 - d}\rho_0.$$
Moreover, we know $\rho_0 > K_0^{-c},$ so
\begin{equation}\label{eq:rho1ineq}
\rho_1 > K_0^{-c}K_1^{1 - d} > K_1^{-d}.
\end{equation}

Next, observe that
\begin{align}
\norm{q_0q_1 \omega_1'} &= \norm{q_0q_1 (\omega_1, \tilde\omega)}\\
&= \norm{q_0q_1\omega_1}+  \norm{q_0q_1\tilde\omega}\\
&< q_1\epsilon_0 + q_0\delta_1\\
&=: \epsilon_1.
\end{align}
Moreover,
\begin{equation}\label{eq:epsilon1Kbound}
\epsilon_1 \leq K_1(\epsilon_0 + \delta_1).
\end{equation}

Now, define $K_2 = K_1^{C_1},$ where $C_1 = d/c$ and observe that \eqref{eq:rho1ineq} implies
$$\rho_1 > K_2^{-c}.$$
Finally, define 
\begin{equation}
\delta_2 = \min_{0 < |k| \leq K_2} \norm{k\cdot \omega_2'}.
\end{equation}

At this point there are two possible scenarios. Either
\begin{equation}\label{eq:CaseN0good}
N_0 \geq K_2^2 \delta_2^{-1}
\end{equation}
or
\begin{equation}\label{eq:CaseN0bad}
N_0 < K_2^2\delta_2^{-1}.
\end{equation}
We will consider \eqref{eq:CaseN0good} first.

Our strategy here is to appeal to Theorem \ref{Thm:Mixed}. Suppose \eqref{eq:CaseN0good} holds. If $N_0 \leq \epsilon_1^{-1} K_2^{-1},$ then we may appeal to our induction hypothesis applied to $\omega'$ with $K_0$ and $\delta_0$ replaced by $K_2$ and $\delta_2,$ respectively, to obtain
\begin{equation}\label{eq:NewScaleConcl}
|L_{N_0} - L_{N_{0,2}}| < K_2^{-c} < K_0^{-c},
\end{equation}
for all $N_{0,2}$ such that $N_0 | N_{0,2}$ and $N_{0,2} \leq \epsilon_1^{-1} K_2^{-1}.$ Now set $N_{0,2}$ as close as possible to $\epsilon_1^{-1} K_2^{-1}.$ If $N_0 > \epsilon_1^{-1}K_2^{-1},$ set $N_{0,2} = N_0.$ 

Now that we have the new length scale $N_{0,2},$ we want to apply Theorem \ref{Thm:Mixed} to $\omega'$ with $N_0$ replaced by $N_{0,2}, \delta = \delta_0, K = K_0,$ and $\kappa = K_0^{-c}.$ It remains to verify the hypothesis of Theorem \ref{Thm:Mixed}.

First, if \eqref{eq:Mixed:hyp1} fails, then
$$\epsilon_0 N_{0,2} > \kappa^C \rho^3 > K_0^{-1},$$
and thus $N_{0,2} > \epsilon_0^{-1}K_0^{-1}$ and we have reached our desired scale, in which case \eqref{eq:NewScaleConcl} is our desired conclusion. Thus, we may assume that \eqref{eq:Mixed:hyp1} holds.

Next, recall that $N_{0,2} \geq N_0,$ so 
\begin{align}
\kappa^{-C}\delta_0^{-1}K_0 &= K_0^{1 + cC} \delta_0^{-1}\\
&\leq K_0^2 \delta_0^{-1}\\
&\leq N_0\\
&\leq N_{0,2}
\end{align}
as long as $c \leq 1/C.$ We also have
\begin{align}
(\rho^{1 + c}\kappa)^{-C} q_0 &< (K_0^{-c - c^2}K_0^{-c})^{-C} q_0\\
&= K_0^{(2+c)cC} q_0\\
&< K_0^{(2 + c)cC +1/10}\\
&< K_0
\end{align}
for $c$ sufficiently small (say $c < 1/3C$). Using such a $c$ throughout only changes the exponent of $\kappa$ in the conclusions in the previous sections by an amount proportional to the change we make to $c$ here. It follows that \eqref{eq:Mixed:hyp2} and \eqref{eq:Mixed:hyp3} hold. Theorem \ref{Thm:Mixed} is thus applicable and we obtain
\begin{equation}
|L_{N_1} - L_{N_{0,2}}| < K_0^{-c},
\end{equation}
and thus
\begin{equation}
|L_{N_0} - L_{N_1}| < 2 K_0^{-c},
\end{equation}
for all $N_{0,2} | N_1$ such that 
$$N_1 < \min\set{\kappa^C\frac{\rho^3q_0}{\norm{q_0\omega_1}}, N_{0,2}e^{(K_0/q_0)^c}}.$$
Note that 
$$\kappa^C\frac{\rho^3q_0}{\norm{q_0\omega_1}} > K_0^{-1}\epsilon_0^{-1}$$
and, using \eqref{eq:epsilon1Kbound},
\begin{align}
N_{0,2}e^{(K_0/q_0)^c} &> N_{0,2} e^{K_0^{c 9/10}}\\
&\gtrsim  K_1^{-1}\epsilon_1^{-1} e^{K_0^{c 9/10}} \\
&= K_0^{-20C_1}\frac{1}{q_1\epsilon_0 + q_0 \delta_1} e^{K_0^{c 9/10}}\\
&> 2K_0^{40}(\delta_1+ \epsilon_0)^{-1}\\
&= 2K_1^2(\delta_1 + \epsilon_0)^{-1}\label{eq:187}
\end{align}
If $\delta_1 < \epsilon_0,$ then the right hand side is no less than $K_1^{2}\epsilon_0^{-1} > \epsilon_0^{-1}K_0^{-1},$ and we have reached our desired scale. On the other hand, if $\delta_1 > \epsilon_0,$ then the right hand side is no less that $K_1^2\delta_1^{-1}.$ If
$$K_1^2\delta_1^{-1} > \epsilon_0^{-1}K_0^{-1},$$ 
then we have reached our desired scale length, and we are done. On the other hand, if 
$$K_1^2\delta_1^{-1} < \epsilon_0^{-1}K_0^{-1},$$
then we may repeat our entire argument above with $N_0$ replaced by $N_1 \sim K_1^2\delta_1^{-1}.$ This puts us in the situation where our conclusion is vacuously true, as described at the start of this proof. Thus, either situation leads to our desired conclusion. 

It now remains to consider the case \eqref{eq:CaseN0bad}. We may perform another change of variables by applying another suitable matrix, $B_1 \in SL_d(\Z),$ with entries bounded by $K_2$ so that
\begin{equation}
B_1\omega' = (\omega_1'',\omega_2'') \in \T^{d_1 + 2}\times \T^{d_2 - 2}
\end{equation}
and
\begin{equation}
\norm{q_2\omega_1''} = \delta_2
\end{equation}
for some $q_2 \in \N, q_2 \leq K_2.$ This, as with the first change of variables, decreases the width of the strip for which we have a suitable subharmonic extension to
$$\rho_2 = \rho_1K_2^{1 - d} > K_2^{-d}.$$
We now define $K_3 = K_2^{C_1}$ so that $\rho_2 > K_3^{-c},$ and set
$$\delta_3 = \min_{0 < |k| \leq K_3} \norm{k\cdot\omega_2''}.$$
Now, we are once again in a situation where one of two things must hold. Either
\begin{equation}
N_0 \geq K_3^2/\delta_3
\end{equation}
or
\begin{equation}
N_0 < K_3^2/\delta_3.
\end{equation}
We assume $N_0 \geq K_3^2/\delta_3.$ Indeed, if not, we will perform another change of variables as above. We now, as before, assume $N_0 < \epsilon_2^{-1} K_3^{-1}$ and apply our induction hypothesis to $(\omega_1'',\omega_2'')$ with $K_0$ replaced with $K_3,$ $\delta_0$ replaced by $\delta_3,$ $q_0$ replaced by $q_0q_1q_2,$ and $\epsilon_0$ replaced by 
\begin{align}
\epsilon_2 &= \norm{q_0q_1q_2\omega_1''} \\
&\leq q_2\epsilon_1 + q_0q_1\delta_2\\
&\leq q_2(q_0\delta_1 + q_1\epsilon_0) + q_0q_1\delta_2\\
&\leq K_2(\delta_1 + \delta_2 + \epsilon_0).
\end{align} 
Thus 
\begin{equation}
|L_{N_0} - L_{N_{0,3}}| < K_3^{-c} < K_0^{-c}
\end{equation}
for all $N_{0,3}$ such that $N_0 | N_{0,3}$ and $N_{0,3} < \epsilon_2^{-1}K_3^{-1}.$ Now fix $N_{0,3}$ as close as possible to $\epsilon_2^{-1}K_3^{-1}.$ If $N_{0} > \epsilon_2^{-1}K_3^{-1},$ then we set $N_{0,3} = N_0.$

Now either $\epsilon_0 N_{0,3} > K_0^{-1},$ in which case we have reached our desired scale and there is nothing else to do, or we may apply Theorem \ref{Thm:Mixed}, the hypotheses of which hold using the same argument as we used for $N_{0,2}.$ In the latter case, we obtain
\begin{align}
|L_N - L_{N_{0,3}}| &< K_0^{-c}\\
|L_N - L_{N_0}| &< K_0^{-c}
\end{align}
for all $N_{0,3} | N$ such that
$$N < \min\set{\epsilon_0^{-1}K_0^{-1}, N_{0,3}e^{K_0^c}}.$$
By our choice of $N_{0,3},$ we have
$$N_{0,3}e^{K_0^c} > K_2^{4}\epsilon_2^{-1} > \frac{K_2^2}{\epsilon_0 + \delta_1 + \delta_2}.$$
If $\epsilon_0 + \delta_1 > \delta_2,$ then 
$$N_{0,3}e^{K_0^c} > \frac{K_2^2}{2\epsilon_0 + 2\delta_1},$$ and we have reached the desired scale: either $\epsilon_0 < \delta_1,$ in which case the right hand side is no less than $K_1^2\delta_1^{-1} + N_0,$ or $\epsilon_0 > \delta_1,$ in which case the right hand side is no less than $\epsilon_0^{-1}K_0^{-1}.$ On the other hand, if $\epsilon_0 + \delta_1 \leq \delta_2,$ then
$$N_{0,3}e^{K_0^c} > \frac{K_2^2}{2\delta_2}.$$ In this case, we can repeat all of the preceding using $N_{0,3}$ instead of $N_0.$ Since $N_{0,3} \geq \frac{K_2^2}{2\delta_2},$ we are in the first scenario we considered, and our conclusion follows. 

If $N_0 < K_3^2/\delta_3,$ then, as remarked above, we may perform another change of variables to define $K_4, \delta_4$ and we may repeat the above procedure. Suppose, therefore, that for some $2 \leq j \leq d_2 + 1,$ we may perform $j$ changes of variables as above and obtain $N_0 \geq K_j^2/\delta_j.$ We may set $\delta_{d_2 + 1} = 1.$ The above procedure allows us to define a scale $N_{0,j}$ such that
$$|L_{N_0} - L_{N_{0,j}}| < K_0^{-c}$$
and
$$|L_{N_{0,j}} - L_N| < K_0^{-c}$$
for $N_{0,j} | N$ such that
$$N < \min\set{\epsilon_0^{-1}K_0^{-1}, N_{0,j}e^{K_0^c}}.$$
Either $N$ has reached the desired scale, in which case we are done, or $N$ satisfies $N \geq K_{j - 1}^2/\delta_{j - 1}.$ We may now repeat the entire procedure starting at scale $N$ instead of $N_0,$ and we will reach our desired scale after at most $d_2$ iterations of this argument. 

Finally, we must consider the case where, no matter how many changes of variables we use, we are never in the case where $N_0 \geq K_j^2/\delta_j,$ where $K_j$ and $\delta_j$ are defined inductively as above for $d_2 \geq j \geq 2,$ and $\delta_{d_2 + 1} = 1.$ In this case, we may apply Theorem \ref{Thm:Mixed} with $N_0, \delta_0,$ and $K_0.$ This leads to 
$$|L_{N_0} - L_{N_0'}| < K_0^{-c}$$ 
for $N_0 | N_0'$ and $$N_0' \leq \min\set{\epsilon_0^{-1}K_0^{-1}, N_0e^{(K_0/q_0)^{c}}}.$$ If $N_0e^{(K_0/q_0)^{c}} > \epsilon_0^{-1}K_0^{-1},$ then we have reached our desired scale and we are done. Otherwise, set $N_0' \sim N_0e^{(K_0/q_0)^{c}}.$ At this point, we will suppose that $K_0 > K'(d_2)$ is large enough such that
$$N_0e^{(K_0/q_0)^{c}} > K_{d_2 + 1}^2.$$ 
With this in hand, we may repeat the entire argument starting at scale $N_0',$ and know that we are guaranteed to satisfy $N_0 \geq K_j^2/\delta_j$ for some $2 \leq j \leq d_2 + 1.$

This completes our induction argument.

\end{proof}





We may now prove Lemma \ref{Lem:FinalLemma}.
\begin{proof}[Proof of Lemma \ref{Lem:FinalLemma}]
Apply Theorem \ref{Thm:MainStep} with $\epsilon_0 = 0, \delta_0 = \delta, K_0= K,$ and $N_0 = N.$ This yields
$$|L_N - L_{N'}| < K^{-c}$$ for all $N | N', N' < \infty.$ Taking a limit, $N' \to \infty,$ completes the proof.
\end{proof}

\section{Proof of Continuity}\label{Section:ContArg}
We are now in a position to prove continuity of $L.$

\begin{mylemma}
Suppose $\omega = (\omega_1,...,\omega_d) \in \T^d.$ Let $(A, \omega)$ be an analytic quasiperiodic $M(2,\C)$-cocycle which has an analytic extension to the strip $|\Im(z_j)| < \rho.$ Suppose, moreover, that $\det(A) \not\equiv 0.$ Then $L(A,\omega)$ is jointly continuous in $A$ and $\omega$ for $\omega$ such that $k \cdot \omega \ne 0$ for any $k \in \Z^d\backslash\set{0}.$
\end{mylemma}

\begin{proof}
Fix a cocycle $(A,\omega),$ where $\omega = (\omega_1,...,\omega_d)$ is such that $\norm{k\cdot\omega} \ne 0$ for all $k \in \Z^d, |k| \ne 0.$ Fix $\kappa > 0$ and let $K_0$ be large enough such that $K_0^{-c} < \kappa.$ Set $\delta_0 = \min_{0 < |k| \leq K_0}\set{\norm{k\cdot\omega}} > 0$ and take $N > K_0^2/\delta_0.$ We have, by Lemma \ref{Lem:FinalLemma} with $d_1 = 0,$
$$|L_N(A,\omega) - L(A,\omega)| < C(A)K_0^{-c} < C(A)\kappa.$$
Moreover, for fixed $N,$ we know $L_N(A,\omega)$ is jointly continuous in $A$ and $\omega,$ so for any cocycle $(B,\omega')$ such that $\norm{A - B}$ and $\norm{\omega - \omega'}$ are sufficiently small, we have
$$|L_N(A,\omega) - L_N(B,\omega')| <  \kappa.$$
Finally, for $\omega'$ sufficiently close to $\omega,$ we have $\norm{k\cdot \omega'} > \frac12\delta_0$ for $0 < |k| \leq K_0.$ Thus
$$|L_N(B,\omega') - L(B,\omega')| <  C(A)\kappa.$$ Here we have $C(A)$ by taking $B$ sufficiently close to $A.$
Triangle inequality now yields our conclusion.
\end{proof}

\begin{mylemma}
Suppose $\omega = (\omega_1,...,\omega_d) \in \T^d.$ Let $(A, \omega)$ be an analytic quasiperiodic $M(2,\C)$-cocycle which has an analytic extension to the strip $|\Im(z_j)| < \rho.$ Suppose, moreover, that $\det(A) \not\equiv 0.$ Then $L(A,\omega)$ is continuous in $A$ for any $\omega \in \T^d.$
\end{mylemma}

\begin{proof}
Fix $\omega \in \T^d.$ If $k\cdot \omega \ne 0$ for any $k\in \Z^d\backslash\set{0},$ then continuity in $A$ follows from joint continuity for at such $\omega.$ Thus, it suffices to suppose $k\cdot \omega = 0$ for some $k\in \Z^d\backslash \set{0}.$ 
First, we claim that we may assume that $\omega = (\omega_1,\omega_2) \in \T^{d_1}\times\T^{d_2}$ is such that $\norm{q\omega_1} = 0$ for some $q\in \N$ and $\norm{k\cdot\omega_2} \ne 0$ for all $k\in \Z^{d_2}, |k| \ne 0.$ 
Indeed, suppose $\norm{k\cdot\omega} = 0$ for some $|k| \ne 0.$ We may perform a change of variables, $B_1,$ so that $$B_1(\omega) = (\omega_1,\omega_2) \in \T\times\T^{d - 1},$$ where $\omega_1 \in \Q.$ If $\omega_2 \in \T^{d-1}$ is such that $\norm{k'\omega_2} = 0$ for some $|k'| \ne 0,$ then we may perform another change of variables, $B_2,$ such that $$B_2(\omega_1,\omega_2) = (\omega_1',\omega_2')\in \T^2\times\T^{d-2},$$ where, for some $q, \norm{q\omega_1'} = 0.$ 
We may thus perform consecutive changes of variables until we reach $\omega' = (\omega_1',\omega_2') \in \T^{d_1}\times\T^{d_2}$ where $\norm{q\omega_1'} = 0$ for some $q\in \N$ and $\norm{k\cdot\omega_2'} \ne 0$ for all $k\in \Z^{d_2}, |k| \ne 0.$ Since a change of variables will not change the regularity of the Lyapunov exponent, this proves our reduction claim. 

Now, assuming $$\omega = (\omega_1,\omega_2) \in \T^{d_1}\times\T^{d_2}$$ is such that $\norm{q\omega_1} = 0$ for some $q\in \N$ and $\norm{k\cdot\omega_2} \ne 0$ for all $k\in \Z^{d_2}, |k| \ne 0,$ our conclusion follows from Lemma \ref{Lem:FinalLemma}, continuity of $L_N$ for fixed $N,$ and triangle inequality.
\end{proof}

\appendix

\section{Plurisubharmonic function facts and estimates: the proofs} \label{Appendix}

When $d = 1,$ Fourier coefficient decay follows from an application of the following result from \cite{DuarteKleinBook}.

\begin{mylemma}[\cite{DuarteKleinBook} Lemma 6.7]\label{Lem:OneFreqFourieru}
Suppose $u: \T \to \R$ is a subharmonic function with a subharmonic extension to $|\Im z| < \rho$ such that 
\begin{equation}
\sup_{|\Im z| < \rho/4} u(z) + \norm{u}_{L^2} \leq C.
\end{equation}
Then there exists a constant $C',$ dependent only on $C$ and $\rho,$ such that
\begin{equation}
|\hat u (k) | < C('|k| + 1)^{-1}.
\end{equation}
\end{mylemma}

The multifrequency estimate follows from the 1-frequency estimate.

\begin{mylemma}
Suppose $u: \T^d \to \R$ is a plurisubharmonic function with a plurisubharmonic extension to $|\Im z_j| < \rho$ such that
\begin{equation}
\sup_{|\Im z_j| < \rho/4} u(z) + \max_{1 \leq j \leq d} \sup_{x_i \in \T, i \ne j} \norm{u(x_1,...,x_j,...,x_d)}_{L^2(dx_j)} < C.
\end{equation}
Then there exists a constant $C',$ dependent only on $C$ and $\rho,$ such that
\begin{equation}
\sum_{|k| > K_0} |\hat u (k)|^2 \leq C' K_0^{-1}.
\end{equation}
\end{mylemma}

\begin{proof}
Observe that, for any fixed $x_1,...,x_{j - 1}, x_{j + 1}, ..., x_d,$ we may define $u_j(x_j) = u(x_1,...,x_j,...,x_d)$ and, by assumption, we have
$$\sup_{|\Im z| < \rho/4} u_j(z) + \norm{u_j(x_j)}_{L^2(dx_j)} \leq C.$$
Thus Lemma \ref{Lem:OneFreqFourieru} applies to $u_j$ and we have, for every $j,$
$$|\hat u_j (k_j)| \leq C' (|k_j| + 1)^{-1}.$$
Moreover, the constant $C'$ is independent of $j$ and $x_i, i \ne j.$ It follows that
$$\sum_{|k_j| > K_0} |\hat u_j(k_j)|^2 \leq (C')^2K_0^{-1}.$$
This may be rewritten as
$$\sum_{|k_j| > K_0} |\hat u(x_1,...,x_{j - 1}, x_{j + 1}, ..., x_d)(k_j)|^2 \leq (C')^2 K_0^{-1}.$$
We may now integrate the left hand side in the variables $x_i, i\ne j,$ and apply Parseval's identity to obtain
$$\sum_{|k_j| > K_0, k_i\in \Z, i\ne j} |\hat u(k_1,...,k_{j - 1}, k_j, k_{j + 1},...,k_d)|^2 \leq (C')^2 K_0^{-1}.$$
Since this holds for all $1 \leq j \leq d,$ we have
$$\sum_{k \in \Z^d, |k| > K_0} |\hat u(k)|^2 < (C')^2 K_0^{-1}$$
as desired.
\end{proof}

It now suffices to show that the hypothesis of these previous lemmas hold in our setting.

\begin{mylemma}\label{Lem:CoVL2}
Suppose $(A,\omega)$ is an analytic $M(2,\C)$ cocycle such that $\det(A)$ does not vanish everywhere. Moreover, suppose 
\begin{equation}\label{eq:APPEq}
\max_{1 \leq j \leq d} \sup_{x_i \in \T, i \ne j} \norm{\ln |\det A(x_1,...,x_j,...,x_d)|}_{L^2(dx_j)} < C.
\end{equation}
Then the previous two lemmas apply with $u(x) = L_N(A,x)$ with $C' = C(A)$ independent of $N.$ 
\end{mylemma}

\begin{proof}
Set $u(x) = L_N(A,x).$ It suffices to verify that $u(x)$ satisfies
\begin{equation}
\sup_{|\Im z_j| < \rho/4} u(z) + \max_{1 \leq j \leq d} \sup_{x_i \in \T, i \ne j} \norm{u(x_1,...,x_j,...,x_d)}_{L^2(x_j)} < C.
\end{equation}
Indeed, recall that, for any $A \in M(2,\C)$ with $\det (A) \ne 0,$ we have
$$\norm{A}^2 \geq |\det(A)|.$$
Since we assume $\det(A(x))$ does not vanish everywhere, we know
$$\norm{A(x)}^2 \geq |\det(A(x))|$$
for a.e. $x\in \T^d.$ Moreover, $A(x)$ has an analytic extension, continuous up to the boundary, to $|\Im z_j| < \rho$ for some $\rho > 0,$ such that, for some $M> 0, \norm{A}_\rho < M.$ Thus $$u(x) = L_N(A,x)$$ 
is a plurisubharmonic function with a plurisubharmonic extension to $|\Im z_j| < \rho$ such that $$\max_{|\Im z_j| < \rho/4} u(z) < |\ln M|.$$
Finally, we have
$$u(x) \geq \frac 1 {2N} \sum_{j = 0}^{N - 1} \ln |\det(A(x + j\omega))|$$
by properties of $\det.$ Hence
\begin{align}
\norm{u(x_1,...,x_j,...,x_d)}_{L^2(dx_j)} &\leq \max \set{|\ln M|, \norm{\frac 1 {2N} \sum_{k = 0}^{N - 1}\ln|\det(A(x + k\omega))|}_{L^2(dx_j)}}\\
&\leq \max \set{|\ln M|, \frac 1 {2N} \sum_{k = 0}^{N - 1}\norm{\ln|\det(A(x + k\omega))|}_{L^2(dx_j)}}\\
&\leq \set{|\ln M|, \frac 12|\ln C|}.
\end{align}
It follows that, for some $C,$ depending only on properties of $A,$
$$\max_{|\Im z_j| < \rho/4} u(z) + \max_{1 \leq j \leq d} \sup_{x_i \in \T, i \ne j} \int_\T |u(x_1,...,x_j,...,x_d)|^2 dx_j < C,$$
and we are done.
\end{proof}

It now suffices to ensure that \eqref{eq:APPEq} holds.

\begin{mylemma}[\cite{DuarteKleinBook} Theorem 6.3]
Suppose $A \in C_\rho^\omega(\T^d, \C)$ is such that $\det(A)$ does not vanish identically. Then there exist $\delta = \delta(A) > 0, C = C(A) < \infty,$ and a linear change of coordinates matrix, $M \in SL(d, \Z),$ such that for any $B \in C_\rho^\omega(\T^d, \C)$ such that $\norm{A - B}_\rho < \delta,$ we have $f(x) = \det(B\circ M(x))$ satisfies 
$$\max_{1 \leq j \leq d} \sup_{x_i \in \T, i \ne j} \norm{f(x_1,...,x_j,...,x_d)}_{L^2(dx_j)} < C.$$
\end{mylemma}

This ensures Lemma \ref{Lem:CoVL2} is applicable to the cocycle $A \circ M(x).$ This implies that the argument in our paper actually applies to $A\circ M.$ However, since $M \in SL(d,\Z)$ is a constant matrix, the Lyapunov exponent for $A\circ M$ and $A$ are the same. We may assume, therefore, that $M$ is the identity matrix.

Now we turn our attention to the boosting inequality. We will derive the $d > 1$ estimate from the $d = 1$ estimate. First, we recall a known BMO estimate.

\begin{mylemma}[\cite{DuarteKleinBook} Lemma 6.8]\label{Lem:BMO1Est}
Suppose $u: \T \to \R$ is a subharmonic function with a subharmonic extension to $|\Im z| < \rho$ such that 
\begin{equation}\label{eq:BMOHypo1}
\sup_{|\Im z| < \rho/4} u(z) + \norm{u}_{L^2} \leq C.
\end{equation}
Moreover, suppose 
\begin{equation}\label{eq:BMOHypo2}
\left|\set{x\in \T: |u(x) - \int u(x) dx| > \epsilon_0}\right| < \epsilon_1.
\end{equation}
Then there exists a constant $C',$ dependent only on $C$ and $\rho,$ such that
\begin{equation}\label{eq:BMOConc}
\norm{u}_{BMO(\T)} \leq C'(\epsilon_0 + \epsilon_1^{1/2}).
\end{equation}
\end{mylemma}

\begin{mylemma}
Suppose $u: \T \to \R$ is a subharmonic function with a subharmonic extension to $|\Im z| < \rho$ such that 
\begin{equation}
\sup_{|\Im z| < \rho/4} u(z) + \norm{u}_{L^2} \leq C.
\end{equation}
Moreover, suppose 
\begin{equation}\label{eq:BMOHypo3}
\norm{u - \int u dx}_{L^1(\T)} < \epsilon.
\end{equation}
Then there exists a constant $c,$ dependent only on $C$ and $\rho,$ such that
\begin{equation}\label{eq:BMOConc2}
\left|\set{x \in \T: \left|u(x) - \int_\T u \right| > \epsilon^{1/6}}\right| < e^{c \epsilon^{-1/6}}.
\end{equation}
\end{mylemma}

\begin{proof}
Let $\epsilon_0 = \epsilon^{1/3}$ and $\epsilon_1 = \epsilon^{2/3}.$ Then 
$$\left|\set{x\in \T: |u(x) - \int u(x) dx| > \epsilon_0}\right| < \epsilon_1,$$
and Lemma \ref{Lem:BMO1Est} is applicable. We obtain
\begin{equation}
\norm{u}_{BMO(\T)} \leq C'\epsilon^{1/3}.
\end{equation}
Now recall the John-Nirenberg inequality:
\begin{equation}
\left|\set{x\in \T: \left|f - \int_\T f \right| > \lambda}\right| < c_1 e^{-c_2\frac{\lambda}{\norm{f}_{BMO}}}.
\end{equation}
Setting $f = u$ and $\lambda = \epsilon^{1/6}$ completes our proof.
\end{proof}

We can now apply this in each variable to deduce an analogue for $d > 1$ (c.f. Lemma 1.27 \cite{BourgainContinuity}).

\begin{mylemma}
Suppose $u: \T^d \to \R$ is a plurisubharmonic function with a plurisubharmonic extension to $|\Im z_j| < \rho$ such that
\begin{equation}
\sup_{|\Im z_j| < \rho/4} u(z) + \max_{1 \leq j \leq d} \sup_{x_i \in \T, i \ne j} \norm{u(x_1,...,x_j,...,x_d)}_{L^2(dx_j)} < C.
\end{equation}
Assume, moreover, that 
$$\norm{u - \int_{\T^d} u }_{L^1(\T^d)} < \epsilon.$$
Then there exist constants $c,$ dependent only on $C$ and $\rho,$ and $a = a(d),$ such that
\begin{equation}
\left|\set{x \in \T: \left|u(x) - \int_{\T^d} u\right| > \epsilon^{a}}\right| < e^{c \epsilon^{-a}}.
\end{equation}
\end{mylemma}

\section*{Acknowledgement}
We would like to thank S. Jitomirskaya for many fruitful discussions and for comments on an earlier version of our manuscript. We would also like to thank W. Liu for comments on an earlier version of our manuscript. This research was partially supported by NSF DMS-2052572, DMS-2052899, DMS-2155211, and Simons 681675.

\bibliographystyle{alpha} 
\bibliography{ContinuousL2022}
\end{document}